\documentclass[12pt]{amsart}
\usepackage{amssymb,amsmath}
\usepackage{geometry}    
\usepackage{mathrsfs}

\usepackage{vmargin}
\setmargrb{1in}{1in}{1in}{1in}

\newtheorem{theorem}{Theorem}[section]
\newtheorem{lemma}[theorem]{Lemma}
\newtheorem{proposition}[theorem]{Proposition}
\newtheorem{corollary}[theorem]{Corollary}
\theoremstyle{definition}
\newtheorem{definition}[theorem]{Definition}
\newtheorem{remark}[theorem]{Remark}

\newtheorem{conjecture}[theorem]{Conjecture}

\newcommand{\scp}[1]{\langle#1\rangle}
\newcommand{\presuc}[1]{\mathopen{\prec}#1\mathclose{\succ}}

\begin{document}

\title{On a $\mathbb{Z}$-module connected to approximation theory}

\author{Johannes Schleischitz}


\begin{abstract}
This paper deals with the set of $\alpha\in{\mathbb{R}}$
such that $\alpha \zeta^{n} \bmod 1$ tends to $0$ for a fixed $\zeta\in{\mathbb{R}}$, which
we call $\mathscr{M}_{\zeta}$.
Predominately the case of Pisot numbers $\zeta$ is studied. In this case 
the inclusions $\mathcal{O}_{\mathbb{Q}(\zeta)}\subset\mathscr{M}_{\zeta}\subset\mathbb{Q}(\zeta)$ 
are known. We will show the properties of $\mathscr{M}_{\zeta}$ are connected
to the module structure of the ring of integers $\mathcal{O}_{\mathbb{Q}(\zeta)}$.
We will describe the module structure of $\mathscr{M}_{\zeta}$ and how much $\mathscr{M}_{\zeta}$ 
differs from $\mathcal{O}_{\mathbb{Q}(\zeta)}$. The results besides allow to give some information 
on the shape of integral bases of real number fields.
\end{abstract}

\maketitle

{\footnotesize{ Supported by FWF grant P24828} \\
Institute of Mathematics, Department of Integrative Biology, BOKU Wien, 1180, Vienna, Austria \\

Math subject classification: 11J71, 11J17, 11R33, 11T23   \\
key words: Pisot numbers,  distribution modulo $1$, module theory, 
algebraic number fields, integral basis, discriminant}

\vspace{8mm}

\section{Introduction}

\subsection{Definitions and basic facts on Pisot numbers} \label{ggg}

We start with some definitions concerning representations of real numbers mod $1$.

\begin{definition}
For $x\in{\mathbb{R}}$, denote by $\lfloor x\rfloor\in{\mathbb{Z}}$ the largest integer smaller than or equal to $x$,
and $\lceil x\rceil\in{\mathbb{Z}}$ the smallest integer greater than or equal to $x$. Let further $\{x\}\in{[0,1)}$ be the fractional part of $x$, i.e., $\{x\}=x-\lfloor x\rfloor$. 
Let $\scp{x}=\lfloor x+1/2\rfloor$, which is rounding to the nearest integer (if $\{x\}\neq 1/2$).
Denote with $\Vert x\Vert:=\vert x-\scp{x}\vert\in{[0,1/2]}$ the distance to the nearest integer to $x$. 
If for a sequence $(x_{n})_{n\geq 1}$ we have $\lim_{n\to\infty} \Vert x_{n}\Vert=0$, 
we will say $(x_{n})_{n\geq 1}$ converges to integers.
\end{definition}

When studying the fractional parts of the sequence $(\alpha\zeta^{n})_{n\geq 1}$, 
Pisot numbers are of particular interest. 

\begin{definition}
A Pisot number is real algebraic integer $\zeta>1$ whose conjugates apart from $\zeta$ itself 
all lie strictly inside the unit circle in $\mathbb{C}$. 
The set of Pisot numbers is denoted by $S$, and $S^{\ast}$ denotes the algebraic units
among $S$.
For a number field $K$, let $S_{K}$ be the set of Pisot numbers that generate $K$,
and let $S_{K}^{\ast}$ be the set of algebraic units within $S_{K}$. 
For $\zeta\in{S}$, denote by
$P_{\zeta}$ the minimal polynomial of $\zeta$, which we refer to as the Pisot polynomial of $\zeta$.
\end{definition}

Pisot numbers are known to have special non-generic distribution properties, the most interesting for our 
purposes are summarized in the following theorem. For the proofs and more basic approximation properties
of Pisot numbers see Chapter~5 in~\cite{27}.

\begin{theorem}[Pisot]  \label{pisot}
Pisot numbers have the property $\lim_{n\to\infty}\Vert\zeta^{n}\Vert=0$.
This property characterizes Pisot numbers among all real algebraic numbers.
Even the following stronger assertion holds:
if $\alpha\zeta^{n}$ converges to integers for a real algebraic number $\zeta>1$ and some $\alpha\neq 0$,
then $\zeta\in{S}$ and $\alpha\in{\mathbb{Q}(\zeta)}$.
\end{theorem}

In this paper, for a fixed $\zeta$ we will study the set
of values $\alpha\in{\mathbb{R}}$ such that $\Vert\alpha\zeta^{n}\Vert$ tends to $0$,
with main focus on Pisot numbers $\zeta\in{S}$. Theorem~\ref{pisot} shows that $\alpha=1$ is always a suitable
choice, and $\alpha\in{\mathbb{Q}(\zeta)}$ is a restriction. Our main objective is to extend these
results on this set, in dependence of $\zeta$. We will see that the set is closely connected
to the ring of integers $\mathcal{O}_{\mathbb{Q}(\zeta)}$. In fact, many proofs will rely on well-known facts
from algebraic number theory summarized in Section~\ref{diesektion}. Our study will also allow us to deduce 
results on the structure of rings of integers of real number fields, see Corollary~\ref{gring}.
However, we will see in general the set differs from $\mathcal{O}_{\mathbb{Q}(\zeta)}$,
see Theorem~\ref{einfall}, for instance.

Finally a notation we will use throughout the rest of the paper.

\begin{definition} \label{defff}
For a ring $R$ we will denote by $R[x_{1},x_{2},\ldots,x_{m}]$
the {\em ring} generated by $x_{1},x_{2},\ldots,x_{m}$ with coefficients in $R$.
By $R\presuc{x_{1},x_{2},\ldots,x_{m}}$ we denote the {\em module} generated by 
$x_{1},x_{2},\ldots,x_{m}$ with coefficients in $R$. 
\end{definition}

\subsection{The set of $\alpha$ with $\Vert\alpha\zeta^{n}\Vert\to 0$ for fixed $\zeta$. 
General results.} \label{najordaj}

We want to present results concerning the algebraic structure
of the set of $\alpha$ for which convergence of $(\alpha\zeta^{n})_{n\geq 1}$ to integers 
for fixed $\zeta>1$ occurs, so for convenient writing we define

\begin{definition}
For fixed real $\zeta>1$, denote by $\mathscr{M}_{\zeta}$ the set of values $\alpha\in{\mathbb{R}}$ such 
that $(\Vert \alpha\zeta^{n}\Vert)_{n\geq 1}$ tends to $0$. 
\end{definition}

The first easy observation on the sets $\mathscr{M}_{\zeta}$ one finds is 
\begin{equation} \label{eq:trivinclu}
\mathscr{M}_{\zeta}\subset \mathscr{M}_{\zeta^{N}}
\end{equation}
for all real numbers $\zeta$ and positive integers $N$. In general
the inclusion in~(\ref{eq:trivinclu}) is strict, see Proposition~\ref{properinclu} for cases of proper inclusion. 
However, see Proposition~\ref{eindimfall} for $\zeta\in{S}$ where $\mathscr{M}_{\zeta}=\mathscr{M}_{\zeta^{N}}$
occurs for all $N\geq 1$.           
Some more properties are given in the following proposition. 

\begin{proposition} \label{mymilena}
For fixed real $\zeta>1$, the set $\mathscr{M}_{\zeta}$ is a $\mathbb{Z}$-module. Moreover
it is closed under $\alpha\mapsto Q(\zeta)\alpha$ for all 
\[
Q(X)=a_{-n}X^{-n}+a_{-n+1}X^{-n+1}+\ldots +a_{-1}X^{-1}+a_{0}+a_{1}X+\ldots +a_{m}X^{m}
\]
with $m,n$ non-negative integers and $a_{i}\in{\mathbb{Z}}, -n\leq i\leq m$. Hence
$\mathscr{M}_{\zeta}$ can also be viewed as a $\mathbb{Z}[\zeta]$-module or a $\mathbb{Z}[\zeta,\zeta^{-1}]$-module.
\end{proposition}

\begin{proof}
We have to prove that for integers $t_{1},t_{2},\ldots,t_{l}$ and any
 $\alpha_{1},\alpha_{2},\ldots,\alpha_{l}$ in $\mathscr{M}_{\zeta}$ we have
 $\sum_{j=1}^{l} t_{j}\alpha_{j}\in{\mathscr{M_{\zeta}}}$.
For any $\epsilon>0$, since $\alpha_{j}\in{\mathscr{M}_{\zeta}}$ we can find $n_{0,j}=n_{0,j}(\epsilon)$ such that 
 $\Vert\alpha_{j}\zeta^{n}\Vert\leq \epsilon/(t_{j}l)$ for all $n\geq n_{0,j}$.
Putting $n_{0}:=\max n_{0,j}$, for all $1\leq j\leq l, n\geq n_{0}$ we have
 $\vert (t_{j}\alpha_{j})\zeta^{n}-t_{j}M_{n,j}\vert \leq \epsilon/l$ for integers $M_{n,j}$.
Hence 
\[
\left\vert\sum_{j=1}^{l} (t_{j}\alpha_{j})\zeta^{n}-\sum_{j=1}^{l}t_{j}M_{n,j}\right\vert \leq 
\sum_{j=1}^{l}\left\vert (t_{j}\alpha_{j})\zeta^{n}-t_{j}M_{n,j}\right\vert \leq \epsilon.
\]
Since $\epsilon>0$ was arbitrary and $\sum_{j=1}^{l}t_{j}M_{n,j}\in{\mathbb{Z}}$ the assertion follows.

For the second claim let $Q(X)$ be of the defined form. For every $-n\leq l\leq m$ and 
any $\alpha\in{\mathscr{M}_{\zeta}}$ the monomials $\Vert \alpha\zeta^{l}\cdot \zeta^{n}\Vert$ 
tend to $0$ for $n\to\infty$ by definition due to $\alpha\in{\mathscr{M}_{\zeta}}$. 
The claim is immediate since we have proved above $\mathscr{M}_{\zeta}$ is a $\mathbb{Z}$-module.
\end{proof}

At this point we point out that though we have the module structure 
of $\mathscr{M}_{\zeta}$, in general it does not form a ring. We will
prove a more general assertion in Proposition~\ref{notring}.

\begin{theorem} \label{mymilaana}
For every transcendental $\zeta>1$ the module $\mathscr{M}_{\zeta}$ is either $\{0\}$ or else not 
finitely generated as a $\mathbb{Z}$-module.
\end{theorem}

\begin{proof}
Observe that $\zeta$ being transcendental means that for all $\alpha\in{\mathscr{M}_{\zeta}}\setminus{\{0\}}$
the monomials $\{\alpha \zeta^{n}: n\in{\mathbb{Z}}\}\subset{\mathscr{M}_{\zeta}}$
are $\mathbb{Q}$-linearly independent. If $\mathscr{M}_{\zeta}$ is finitely generated,
say the $\mathbb{Z}$-span of
some finite generating system $\alpha_{1},\alpha_{2},\ldots,\alpha_{l}$, it is contained in
a $\mathbb{Q}$-vector space of dimension at most $l$. Hence by the above
observation $\mathscr{M}_{\zeta}=\{0\}$ follows.  
\end{proof} 

Now we turn to the case of algebraic $\zeta$.
By Theorem~\ref{pisot}, if $\zeta$ is algebraic, then $\mathscr{M}_{\zeta}=\{0\}$ unless $\zeta$
is a Pisot number. Hence we may restrict to $\zeta\in{S}$.
In this case our results will rely on classic results from algebra, mostly algebraic number theory.
The following Section~\ref{diesektion} summarizes all that we will later refer to, we continue with new results
in the Section~\ref{dritt}.  

\section{Preliminary results from algebra and algebraic number theory} \label{diesektion}

\subsection{Facts from algebra}
Our study of $\mathscr{M}_{\zeta}$ relies on basic properties of polynomial roots and symmetric polynomials.

\begin{definition}  \label{einfach}
For $k\geq 1$ variables $\underline{X}=(X_{1},X_{2},\ldots ,X_{k})$, the {\em elementary symmetric polynomials}  
are given by
\[
 \mu_{k,1}(\underline{X}):=\sum_{1\leq j\leq k} X_{j}, 
\quad \mu_{k,2}(\underline{X}):=\sum_{1\leq i<j\leq k}X_{i}X_{j}, \quad \ldots, 
\quad \mu_{k,k}(\underline{X}):=\prod_{1\leq j\leq k} X_{j}.
\]
\end{definition}

\begin{theorem} [Vieta]\label{vieta}
Let $P(X)=X^{k}+a_{k-1}X^{k-1}\ldots +a_{0}$ be a polynomial with integer coefficients 
and roots $\underline{\zeta}=(\zeta_{1},\ldots,\zeta_{k})$ counted with multiplicity. 
Then with $\mu_{.,.}$ from Definition~\ref{einfach} we have $a_{j}=(-1)^{k-j} \mu_{k,k+1-j}(\underline{\zeta})$.
\end{theorem}

\begin{theorem}  \label{elemsatz}
 Every symmetric polynomial $P\in{\mathbb{Z}[X_{1},X_{2},\ldots,X_{k}]}$
is a polynomial with integer coefficients in the elementary symmetric polynomials.
For a polynomial $Q(X)=X^{k}+a_{k-1}X^{k-1}+\cdots +a_{0}\in{\mathbb{Z}[X]}$ with roots $x_{1},x_{2},\ldots ,x_{k}$,
the number $d_{n}:= \sum_{j=1}^{k} x_{j}^{n}$ is an integer for any non-negative integer $n$, and
if $\zeta:=x_{1}\in{S}$ then $\scp{\zeta^{n}}=d_{n}$ for $n\geq n_{0}$ sufficiently large.
Moreover, {\upshape(}putting $a_{k}=1${\upshape)} we have 
the linear recurrence relation $\sum_{j=0}^{k} a_{j}d_{j+n}=0$.
Thus, the reduction of $(d_{n})_{n\geq 1}$ modulo any integer $M\geq 2$ is periodic
and hence the reduction of $(\scp{\zeta^{n}})_{n\geq 1}$ modulo $M$ is ultimately periodic.  
\end{theorem} 

\begin{proof} 
See \cite{11} for a proof of the first part. The claim that $d_{n}\in{\mathbb{Z}}$ is immediate by the first assertion
and the fact that the elementary symmetric polynomials in $x_{1},x_{2},\ldots,x_{k}$
are integers as well, as they are (up to sign) just the coefficients of $P$ by Vieta Theorem.
For $\zeta\in{S}$ we have $\scp{\zeta^{n}}=d_{n}$ for $n\geq n_{0}$,
since the remaining power sum $\sum_{j=2}^{k} x_{j}^{n}$ converges to $0$ as $n\to\infty$.
Concerning the recurrence relation, note that $\sum_{j=0}^{k} a_{j}x_{i}^{j}=0$ for $1\leq i\leq k$ by definition,
so multiplying these relations with $x_{i}^{n}$ and summing over $i$ gives the assertion. Finally,
recurrence sequences are obviously periodic when reduced modulo $M$ by the finiteness of residue classes. 
\end{proof}

Some basic properties of finite fields are summarized.

\begin{definition} 
For prime $q$ we denote $\mathbb{Z}_{q}$ the field with $q$ elements and $\mathbb{F}_{q^{h}}$ 
its extension fields of finite dimension (such that $\mathbb{Z}_{q}=\mathbb{F}_{q}$ if $h=1$).
\end{definition}
\begin{theorem} \label{polvoll}
Let $\mathbb{F}_{q^{h}}$ be a finite field. Then we have

\begin{enumerate}
\item[(i)] $\mathbb{F}_{q^{h}}$ is separable.
\item[(ii)] every function $f:\mathbb{F}_{q^{h}}\mapsto \mathbb{F}_{q^{h}}$
there is a polynomial $P_{f}\in{\mathbb{F}_{q^{h}}[X]}$ of degree at most $q^{h}$ such that $P_{f}=f$ as functions. 
\end{enumerate}
\end{theorem}

Assertion~(i) is Corollary on page~190 in Chapter~7 in~\cite{lang},
(ii) is an immediate consequence of Lagrange's interpolation formula.
We recall some facts on module structure and discriminants.

\begin{theorem} \label{strucsatz}
Every finitely generated module $M$ over a principal ideal domain $R$ is isomorphic to a unique module of the form
\[
R^{f}\oplus R/(d_{1})\oplus R/(d_{2})\oplus \ldots \oplus R/(d_{m})
\]
where $d_{i}\neq 0, (d_{i})\neq R$ and $d_{i}\vert d_{i+1}$ for $f\geq 0$ an integer $d_{i}\in{R\setminus{\{0\}}}$.
In particular, if $M$ is additionally torsion free, then $M\cong R^{f}$ is free and any free submodule
has dimension at most $f$.  
\end{theorem} 

See Theorem~6.12 in~\cite{hungerford}.

\begin{definition}
For a polynomial 
$P(X)=a_{k}X^{k}+a_{k-1}X^{k-1}+\cdots+a_{0}\in{\mathbb{Z}[X]}$,
the {\em discriminant $\Delta(P)$} is defined by
\[
\Delta(P) = a_{k}^{2k-2}\prod_{i< j} (x_{i}-x_{j})^{2},
\]
where $x_{1},x_{2},\ldots,x_{k}$ are
the roots of $P$ in $\mathbb{C}$. Theorem~\ref{elemsatz} implies $\Delta(P)\in{\mathbb{Z}}$.
\end{definition}

In Theorem~\ref{okkident} we will utilize an alternative way to compute the discriminant. 

\begin{theorem} \label{discoduck}
The discriminant of a monic polynomial $P$ with roots $x_{1},x_{2},\ldots,x_{k}$ 
is the square of the determinant of the Vandermonde matrix
\[
B=\left( \begin{array}{cccccccccc}
1 & x_{1} & x_{1}^{2} & \cdots & x_{1}^{k-1}      \\
1 & x_{2} & x_{2}^{2} & \cdots & x_{2}^{k-1} \\
\vdots & \vdots  & \vdots & \vdots & \vdots  \\
1 & x_{k} & x_{k}^{2} & \ldots & x_{k}^{k-1}  \\
\end{array} \right), 
\]
that is, $\Delta(P)=\det(B)^{2}$.
\end{theorem}

For the proof of this identity, see Theorem~8 page~26 in~\cite{marcus}.

Finally we state a well-known fact on the factorization of
reductions of a polynomial over prime fields, which 
can be derived as a consequence of Corollary~2 (The discriminant theorem) in Chapter~6, paragraph~2 on
page~157 in~\cite{102}.

\begin{theorem}  \label{diskriminante}
Let $P\in{\mathbb{Z}[X]}$. 
For any rational prime $q$ denote by $P_{q}$ the reduction of $P$ over the field with $q$ elements
$\mathbb{Z}_{q}$, i.e., all coefficients of $P$ reduced mod $q$. Then $P_{q}$ has
at least one multiple root over its finite splitting field $\mathbb{F}_{q^{h}}$ 
if and only if $q\vert \Delta(P)$.   
\end{theorem}

Roughly speaking, Theorem~\ref{diskriminante} says that  
$P$ is inseparable over precisely those primes dividing $\Delta(P)$.

\subsection{Definitions and facts from algebraic number theory}

Most results of this section can be found in~\cite{102}.

\begin{definition}
An algebraic number field $K$ is a finite dimensional extension of $\mathbb{Q}$. 
By $\mathcal{O}$ we denote the algebraic integers and by $\mathcal{O}_{K}=K\cap \mathcal{O}$
we denote the algebraic integers of $K$. Next, let $\mathcal{O}^{\ast}$ and $\mathcal{O}_{K}^{\ast}$ 
be the units in $\mathcal{O}$ and $\mathcal{O}_{K}$, respectively.
\end{definition}

\begin{theorem} \label{maximalordnung}
For any number field $K$, the set $\mathcal{O}_{K}$ is actually a ring. 
If $K$ is of dimension $[K:\mathbb{Q}]=k$, the ring of integers $\mathcal{O}_{K}$
is the maximum order of $K$, i.e., the maximum $\mathbb{Z}$-submodule $\mathscr{N}$
of $K$ of dimension $k$ with the property that $\mathscr{N}$ is also a ring. 
\end{theorem}

See~\cite{reiner} for a proof and a deeper study on orders. Another basic result on $\mathcal{O}_{K}$ is the following.

\begin{theorem} \label{nusskn}
The ring of integers $\mathcal{O}_{K}$ of an algebraic number field $K$ 
is a free $\mathbb{Z}$-module  of dimension $[K:\mathbb{Q}]$.
Any generating system is called an integral basis of $\mathcal{O}_{K}$. Two integral bases
differ by multiplication of a matrix $B\in{\mathbb{Z}^{k\times k}}$ with $\vert\det(B)\vert=1$.
\end{theorem}

Proposition~2.5 Chapter~2 paragraph~2 page~57 in~\cite{102}. Some facts factorization in $\mathcal{O}_{K}$
are summarized in the following Theorem~\ref{dastheorem}. 

\begin{theorem} \label{dastheorem}
Let $K$ be a number field with ring of integers $\mathcal{O}_{K}$. Then the following assertions hold.
\begin{enumerate}
\item[(i)] the quotient field of $\mathcal{O}_{K}$ coincides with $K$. More generally,
any element of $K$ can be written as $\frac{y}{b}$ with $y\in{\mathcal{O}_{K}}$ and $b\in{\mathbb{Z}}$.
\item[(ii)] In $\mathcal{O}_{K}$,
there is a unique factorization of ideals of $\mathcal{O}_{K}$ into prime ideals 
of $\mathcal{O}_{K}$.
\item[(iii)] For any $p,q$ distinct primes in $\mathbb{Z}$, 
the ideals $(p)=p\mathcal{O}_{K}$ and $(q)=q\mathcal{O}_{K}$
have no common prime ideal factor in their prime ideal factorization.
\end{enumerate}
\end{theorem}

The result~(i) is immediate by the third fact stated on page~9 in Chapter~1.1 in \cite{swinnerton},
it is also the content of Lemma~2.13 in \cite{hall}. 
Combination of Theorem~1.4 and Theorem~1.9 in~\cite{102} yields~(ii). Finally,~(iii)
is a simple special case of Proposition~4.2 in~\cite{102} on integral ring extensions and
could be inferred directly very easily.

\begin{theorem} \label{embed}
For any algebraic number field $K$ of degree $[K:\mathbb{Q}]=k$ there exist exactly $k$ monomorphisms 
$\sigma_{1},\sigma_{2},\ldots ,\sigma_{k}$ mapping $K\mapsto \mathbb{C}$. 
The product $\prod_{1\leq i\leq k} \sigma_{i}(x)\in{\mathbb{Z}}$ is called the norm $N_{K/\mathbb{Q}}(x)$ 
and the sum $\sum_{1\leq i\leq k} \sigma_{i}(x)\in{\mathbb{Z}}$ the trace of $x\in{K}$. 
\end{theorem}

See the introduction of Chapter~2, pages~45, 46 in~\cite{102}.

\begin{definition}
For an algebraic number field $K$ of degree $[K:\mathbb{Q}]=k$ and 
$\underline{\eta}:=(\eta_{1},\eta_{2},\ldots,\eta_{k})\in{K^{k}}$ the discriminant of
$\underline{\eta}$ is defined as $d_{K}(\underline{\eta})=\det(A)^{2}$ for the 
matrix $A\in{\mathbb{C}^{k\times k}}$ with entries $A_{i,j}=\sigma_{i}(\eta_{j})$
for the embeddings $\sigma_{i}:K\mapsto \mathbb{C}$ from Theorem~\ref{embed}.
The {\em discriminant $d_{K}$ of the number field $K$} is defined as the discriminant of
an integral basis of $\mathcal{O}_{K}$ (which does not depend on the choice of the integral basis
by the last assertion of Theorem~\ref{nusskn}). 
\end{definition}

\begin{theorem}  \label{narre}
The discriminant of a number field $K$ is an integer. Moreover, $\vert d_{K}\vert=1$ if and only if $K=\mathbb{Q}$.
\end{theorem}

This is an immediate consequence of the famous Minkowski bound for the discriminant of
a number field, see Theorem~2.10 in~\cite{102}. 

There is a close connection between polynomial and field discriminants.

\begin{theorem} \label{narrer}
Let $\theta$ be a root of a monic irreducible polynomial $P(X)\in{\mathbb{Z}[X]}$,
and $K=\mathbb{Q}(\theta)$ be the field generated by $\theta$
and $\mathcal{O}_{K}$ its ring of integers. Further let
$\mathbb{Z}[\theta]$ be the subring of $\mathcal{O}_{K}$ generated by $\theta$. Then the discriminant of $P$ 
is the discriminant of the lattice $\mathbb{Z}[\theta]$.
So $\Delta(P)/d_{K}$ is the square of the index $[\mathcal{O}_{K}:\mathbb{Z}[\theta]]$
of $\mathbb{Z}[\theta]$ in $\mathcal{O}_{K}$.
In particular $\Delta(P)/d_{K}$ is square number and $\vert\Delta(P)\vert\geq \vert d_{K}\vert$
with equality if and only if $\mathcal{O}_{K}=\mathbb{Z}[\theta]$. 
\end{theorem}

See Section~III.3 of~\cite{1000}.    
For quadratic number fields integral bases and discriminants are known, see 
Theorem~2.6 in~\cite{102}.
 
\begin{theorem} \label{kopernik}
Let $d\in{\mathbb{Z}}$ be square-free. If $d\equiv 1 \bmod 4$, we have
$\mathcal{O}_{\mathbb{Q}(\sqrt{d})}=\mathbb{Z}+\frac{1+\sqrt{d}}{2}\mathbb{Z}$ and $d_{K}=d$, if 
else $d\equiv 2,3 \bmod 4$ we have $\mathcal{O}_{\mathbb{Q}(\sqrt{d})}=\mathbb{Z}+\sqrt{d}\mathbb{Z}$
and $d_{K}=4d$.
\end{theorem}

Some facts on $\mathcal{O}_{K}^{\ast}$ for quadratic number fields $K$ are summarized in the following.

\begin{theorem} \label{einheitensatz}
In a quadratic number field $\mathbb{Q}(\sqrt{d})$ with square-free $d>0$, 
the units are of the form $\pm(A_{n}+B_{n}\sqrt{d})=\pm(A_{1}+B_{1}\sqrt{d})^{n}$ for 
$n\in{\mathbb{Z}},A_{n},B_{n}\in{\mathbb{Q}}$ and a fundamental unit $A_{1}+B_{1}\sqrt{d}$.
\begin{itemize}
\item[(i)]
If $d\not\equiv 5\bmod 8$, then all $A_{n},B_{n}$ are integers, 
satisfying $A_{n}^{2}-dB_{n}^{2}=\pm 1$. 
\item[(ii)]
In case of $d\equiv 5\bmod 8$, there are additionally possibly units $\zeta=\pm(A_{n}+B_{n}\sqrt{d})$ 
with $A_{n}=\frac{2L_{1}+1}{2}, B_{n}=\frac{2L_{2}+1}{2}$
for $L_{1},L_{2}$ integers, satisfying $A_{n}^{2}-dB_{n}^{2}=\pm 4$.
\end{itemize}
In any case, $\zeta\in{\mathcal{O}_{\mathbb{Q}(\sqrt{d})}^{\ast}}$ 
implies $\zeta^{3}=\pm(A_{n}+B_{n}\sqrt{d})^{3}=\pm(A_{3n}+B_{3n}\sqrt{d})$ 
has $A_{3n},B_{3n}$ integers. In particular, for any $d$
there are infinitely many units $A_{n}+B_{n}\sqrt{d}$ with $A_{n},B_{n}$ integers,
satisfying $A_{n}^{2}-dB_{n}^{2}=\pm 1$.
\end{theorem}

See Chapter~11 in \cite{alaca}, in particular page~290. Finally an important note on Pisot numbers in
a given number field. 

\begin{theorem} \label{reellfeld}
Any real number $K$ field contains Pisot numbers $\theta$ such that $K=\mathbb{Q}(\theta)$.
If $K\neq \mathbb{Q}$,
some of these numbers are algebraic units. Equivalently, $\emptyset\neq S_{K}^{\ast}\subsetneq S_{K}$.
\end{theorem}
The above theorem is Theorem~5.2.2 in~\cite{27}. The proof uses Minkowski's lattice point Theorem
and is thus not constructive.

\section{The set $\mathscr{M}_{\zeta}$ for Pisot numbers} \label{dritt}

\subsection{Preliminaries} By Proposition~\ref{mymilena} and Theorem~\ref{pisot} we have 
\begin{equation}  \label{eq:klumheidi}
\mathbb{Z}[\zeta]\subset \mathbb{Z}[\zeta,\zeta^{-1}]\subset \mathscr{M}_{\zeta}\subset \mathbb{Q}(\zeta).
\end{equation}
Similar to the proof of $1\in{\mathscr{M}_{\zeta}}$, i.e., $\Vert \zeta^{n}\Vert\to 0$,
it is not hard to see that any algebraic integer $\lambda\in{\mathcal{O}_{\mathbb{Q}(\zeta)}}$
is in $\mathscr{M}_{\zeta}$ as well, as mentioned in Section~5.4 in~\cite{27}. 
Indeed, it can be easily inferred by looking at the trace of $\lambda\zeta^{n}$, 
which is an integer by Theorem~\ref{embed}, and using the Pisot property.
Proposition~\ref{mymilena} further implies 
$\mathcal{O}_{\mathbb{Q}(\zeta)}[\zeta,\zeta^{-1}]\subset \mathscr{M}_{\zeta}$. 
Moreover, again since $\zeta\in{\mathcal{O}_{\mathbb{Q}(\zeta)}}$ and 
$\mathcal{O}_{\mathbb{Q}(\zeta)}$ is a ring,
$\mathcal{O}_{\mathbb{Q}(\zeta)}$ contains $\mathbb{Z}[\zeta]$ and we have
$\mathcal{O}_{\mathbb{Q}(\zeta)}=\mathcal{O}_{\mathbb{Q}(\zeta)}[\zeta]$. The latter implies 
$\mathcal{O}_{\mathbb{Q}(\zeta)}[\zeta,\zeta^{-1}]=\mathcal{O}_{\mathbb{Q}(\zeta)}[\zeta^{-1}]$.
By combining these facts,~(\ref{eq:klumheidi}) can be refined to  
\begin{eqnarray} 
\mathbb{Z}[\zeta]\subset \mathcal{O}_{\mathbb{Q}(\zeta)}\subset \mathcal{O}_{\mathbb{Q}(\zeta)}[\zeta^{-1}]
&\subset&\mathscr{M}_{\zeta}\subset \mathbb{Q}(\zeta),   \label{eq:ungarturm}    \\
\mathbb{Z}[\zeta]\subset \mathbb{Z}[\zeta,\zeta^{-1}]\subset \mathcal{O}_{\mathbb{Q}(\zeta)}[\zeta^{-1}]
&\subset&\mathscr{M}_{\zeta}\subset \mathbb{Q}(\zeta).  \label{eq:ungerturm}
\end{eqnarray}
So the ring $\mathcal{O}_{\mathbb{Q}(\zeta)}[\zeta^{-1}]$ can be interpreted as the largest
trivial subset of $\mathscr{M}_{\zeta}$. In general there is no equality
$\mathcal{O}_{\mathbb{Q}(\zeta)}[\zeta^{-1}]=\mathscr{M}_{\zeta}$, see Corollary~\ref{nichtgleich}.
 
Note that~(\ref{eq:ungarturm}) and~(\ref{eq:ungerturm}) cannot be reduced to a single inclusion,
since in general $\mathcal{O}_{\mathbb{Q}(\zeta)}\nsubseteq \mathbb{Z}[\zeta,\zeta^{-1}]$ and
$\mathcal{O}_{\mathbb{Q}(\zeta)}\nsupseteq \mathbb{Z}[\zeta,\zeta^{-1}]$.  
Indeed, for $\mathcal{O}_{\mathbb{Q}(\zeta)}\nsupseteq \mathbb{Z}[\zeta,\zeta^{-1}]$
it suffices to take any $\zeta\in{S\setminus S^{\ast}}$, for in this case by definition
$\zeta^{-1}\notin \mathcal{O}_{\mathbb{Q}(\zeta)}$. See also Proposition~\ref{selchfleisch}.
On the other hand, the Pisot root $\zeta=2+\sqrt{5}$ of $X^{2}-4X-1$ has $\zeta^{-1}=-(2-\sqrt{5})$
so every element of $\mathbb{Z}[\zeta,\zeta^{-1}]$ is of the form $\sqrt{5}\mathbb{Z}+\mathbb{Z}$.
However, 
$\mathcal{O}_{\mathbb{Q}(\zeta)}=((1+\sqrt{5})/2)\mathbb{Z}+\mathbb{Z}$ by virtue of Theorem~\ref{kopernik}, so 
$\mathcal{O}_{\mathbb{Q}(\zeta)}\nsubseteq \mathbb{Z}[\zeta,\zeta^{-1}]$. 
Moreover, the impression suggested by~(\ref{eq:ungarturm}),(\ref{eq:ungerturm}), that $\mathscr{M}_{\zeta}$
could be a $\mathcal{O}_{\mathbb{Q}(\zeta)}$-module let alone a $\mathcal{O}_{\mathbb{Q}(\zeta)}[\zeta^{-1}]$-module
or a ring, is in general false. This is proved in Proposition~\ref{notring}.

For its proof we want to point out the following fact based on Theorem~\ref{elemsatz}, which
will be carried out implicitly many times in following proofs as well.
For $N\in{\mathbb{Z}},\beta\in{\mathbb{R}}$ if $\beta/N\in{\mathscr{M}_{\zeta}}$ 
then $N\vert\scp{\beta\zeta^{n}}$ or equivalently $N\vert\scp{\beta\scp{\zeta^{n}}}$ 
for all $n\geq n_{0}$. In the notation of Theorem~\ref{elemsatz}
for $\zeta\in{S}$ the latter is equivalent to $N\vert \scp{\beta d_{n}}$.

\begin{proposition}  \label{notring}
In general, for $\zeta\in{S}$, we have that $\mathscr{M}_{\zeta}$ is not an $\mathcal{O}_{\mathbb{Q}(\zeta)}$-module.
In particular, in general $\mathscr{M}_{\zeta}$ is no $\mathcal{O}_{\mathbb{Q}(\zeta)}[\zeta^{-1}]$-module
and not a ring.
\end{proposition}

\begin{proof}
The Pisot root $\zeta=2+\sqrt{5}$ of $P_{\zeta}=X^{2}-4X-1$ represents a counterexample.
The recurrence relation from Theorem~\ref{elemsatz} 
shows that all of the numbers $\scp{\zeta^{n}}$ are even, so $1/2\in{\mathscr{M}_{\zeta}}$.
On the other hand, by Theorem~\ref{kopernik} we have 
$(1+\sqrt{5})/2\in{\mathcal{O}_{\mathbb{Q}(\zeta)}}$.
If $\mathscr{M}_{\zeta}$ would be a $\mathcal{O}_{\mathbb{Q}(\zeta)}$-module, then it would contain 
$(1/2)\cdot((1+\sqrt{5})/2)=(1+\sqrt{5})/4$.
However, similar to Theorem~\ref{elemsatz} it is easily shown that for $a=1+\sqrt{5}$, the integers 
$\scp{a\zeta^{n}}$ satisfy the same linear recurrence relation as $\scp{\zeta^{n}}$, i.e., 
$\scp{a\zeta^{n+2}}=4\scp{a\zeta^{n+1}}+\scp{a\zeta^{n}}$, for sufficiently large $n$ too.
One further readily checks $\scp{a\zeta^{n}}\equiv 2 \bmod 4$ actually for all $n\geq 1$, so 
$(1+\sqrt{5})/4\notin{\mathscr{M}_{\zeta}}$.
\end{proof}

See Corollary~\ref{dazwischen} for a case where $\mathscr{M}_{\zeta}$ is a 
$\mathcal{O}_{\mathbb{Q}(\zeta)}[\zeta^{-1}]$-module but no ring and 
Corollary~\ref{wannring} for a large class of $\zeta$ for which $\mathscr{M}_{\zeta}$ is no ring. 
For completeness, the following Proposition~\ref{properinclu}
gives an easy example of with strict inclusion in~(\ref{eq:trivinclu}).
We point out that it
provides examples for both cases $\zeta\in{S^{\ast}}$ and $\zeta\in{S\setminus S^{\ast}}$. 

\begin{proposition} \label{properinclu}
If $d\equiv 1 \bmod 4$, then there exists $\zeta\in{S_{\mathbb{Q}(\sqrt{d})}}$ of the form
$\zeta=(A+B\sqrt{d})/2$ with $A,B$ odd, and for any such $\zeta$ we have
$1/2\notin{\mathscr{M}_{\zeta}}$, but $1/2\in{\mathscr{M}_{\zeta^{3}}}$. In particular,
there are $\zeta\in{S^{\ast}}$ and $\zeta\in{S\setminus S^{\ast}}$ with strict inclusion 
in~{\upshape(\ref{eq:trivinclu})} for all $N$ divisible by three.
\end{proposition}  

\begin{proof}
The existence is easy to see, any choice of odd $A$ leads to some odd $B$. The rest is
similar to the proof of Proposition~\ref{notring}, looking at $\scp{\zeta^{n}}$ and $\scp{\zeta^{3n}}$ mod $2$,
using the fact $\zeta^{3}\in{\sqrt{d}\mathbb{Z}+\mathbb{Z}}$,
see Theorem~\ref{einheitensatz}. The second claim for units follows taking $d\equiv 5 \bmod 8$ with units
of this form, see~{\upshape(ii)} in Theorem~\ref{einheitensatz}, and the existence of
non-units is obvious.
\end{proof}

In the sequel our objective is for $\zeta\in{S}$ to gain information on the following questions.

\begin{itemize}
\item Description of the algebraic structure of $\mathscr{M}_{\zeta}$.
Predominately, is $\mathscr{M}_{\zeta}$ are finitely generated or free as a $\mathbb{Z}$-module,
$\mathbb{Z}[\zeta]$-module or $\mathbb{Z}[\zeta,\zeta^{-1}]$-module? In which cases
is $\mathscr{M}_{\zeta}$ a ring?
\item Considering a fixed real number field $K$, by~(\ref{eq:ungarturm}),(\ref{eq:ungerturm})
and since clearly $S_{K}\subset S\cap K$ and $S_{K}\neq \emptyset$ by Theorem~\ref{reellfeld}, we have

\begin{equation} \label{eq:unnuetzlich}
\bigcap_{\zeta\in{S\cap K}}\mathscr{M}_{\zeta}\subset \bigcap_{\zeta\in{S_{K}}}\mathscr{M}_{\zeta}\subset
\mathcal{O}_{K}\subset \bigcup_{\zeta\in{S_{K}}}\mathscr{M}_{\zeta}
\subset \bigcup_{\zeta\in{S\cap K}} \mathscr{M}_{\zeta}\subset K.
\end{equation}
Which of these inclusions is strict?
\item Description of the nontrivial elements in $\mathscr{M}_{\zeta}$, i.e.,
elements in $\mathscr{M}_{\zeta}\setminus{\mathcal{O}_{\mathbb{Q}(\zeta)}[\zeta^{-1}]}$
\item What can the results tell us about the structure of $\mathcal{O}_{\mathbb{Q}(\zeta)}$?
\end{itemize}

Note that a generalized question of the second last point above, namely
finding non-trivial elements in $\mathscr{M}_{\zeta}\setminus{\mathcal{O}_{\mathbb{Q}(\zeta)}},
\mathscr{M}_{\zeta}\setminus{\mathbb{Z}[\zeta,\zeta^{-1}]}, \mathscr{M}_{\zeta}\setminus{\mathbb{Z}[\zeta]}$
and deciding which of the inclusions in~(\ref{eq:ungarturm}),(\ref{eq:ungerturm}) is proper,
is worth consideration. In Section~\ref{letztmals} we will deal with this question. 
However, apart from the one-dimensional case $k=1$ in Proposition~\ref{eindimfall},
this will not be the focus of the present paper.

\subsection{The case $k=1$.} \label{eindimensionul}
To give an impression, we can readily answer the questions in the case that $\zeta$ is a Pisot number of degree $1$,
i.e., an integer greater than $1$. In this case the radical of $\zeta$ contains all the information.
\begin{definition}
For an integer $N=\pm\prod_{j=1}^{l} p_{j}^{\alpha_{j}}$ denote by $\rm{rad}(N)=\prod_{j=1}^{l} p_{j}$ its radical.
\end{definition}
For $\zeta\in{\mathbb{Z}}$, obviously $\mathscr{M}_{\zeta}$ consists
precisely of rationals $\frac{A}{B}, (A,B)=1$ with $B$ containing only prime factors of $\zeta$,
or in other words $\rm{rad}(B)\vert\rm{rad}(\zeta)$. Thus it is easy to answer the questions
and to provide some additional information.

\begin{proposition} \label{eindimfall}
For any $\zeta\in{\mathbb{Q}\cap S}$, i.e., an integer greater one, we have
\begin{equation}  \label{eq:eindimen}
\mathbb{Z}=\mathbb{Z}[\zeta]=\mathcal{O}_{\mathbb{Q}(\zeta)}\subsetneq \mathbb{Z}[\zeta,\zeta^{-1}]
=\mathcal{O}_{\mathbb{Q}(\zeta)}[\zeta^{-1}]=\mathbb{Z}[1/\rm{rad}(\zeta)]=\mathscr{M}_{\zeta}\subsetneq 
\mathbb{Q}(\zeta)=\mathbb{Q}.
\end{equation}
It follows that $\mathscr{M}_{\zeta}$ is neither finitely generated nor free as a $\mathbb{Z}$-module
or $\mathbb{Z}[\zeta]$-module, and one dimensional as a $\mathbb{Z}[\zeta,\zeta^{-1}]$-module.
Moreover, for $K:=\mathbb{Q}$ we have
 \begin{equation}  \label{eq:rafanadal}
\bigcap_{\zeta\in{S\cap K}}\mathscr{M}_{\zeta}= \bigcap_{\zeta\in{S_{K}}}\mathscr{M}_{\zeta}=
\mathcal{O}_{K}=\mathbb{Z}\subsetneq \bigcup_{\zeta\in{S_{K}}}\mathscr{M}_{\zeta}
= \bigcup_{\zeta\in{S\cap K}} \mathscr{M}_{\zeta}=K=\mathbb{Q}.
\end{equation}
Since $\mathscr{M}_{\zeta}$ in {\upshape(\ref{eq:eindimen})} only depends 
on $\rm{rad}(\zeta)=\rm{rad}(\zeta^{N})$ for all $N\geq 1$,
in the one-dimensional case we always have equality in {\upshape(\ref{eq:trivinclu})}. Thus the 
map $\chi:=\zeta\longmapsto \mathscr{M}_{\zeta}$ from the set $S$ to $2^{\mathbb{R}}$ is not one-to-one. 
Furthermore, for any real number field $K$, since it contains $\mathbb{Q}$, 
{\upshape(\ref{eq:eindimen})} easily implies
\[
\bigcap_{\zeta\in{S\cap K}} \mathscr{M}_{\zeta}= \mathbb{Z}, 
\qquad \bigcup_{\zeta\in{S\cap K}} \mathscr{M}_{\zeta}\supset \mathbb{Q}.
\]
\end{proposition}

Note that Theorem~\ref{reellfeld} implies that $K=\mathbb{Q}$ 
is the only number field with $S^{\ast}\cap K=\emptyset$. 
If we would allow $\zeta=1$ to be a Pisot number, which is somehow reasonable,
there would be equality in $\mathcal{O}_{\mathbb{Q}(\zeta)}\subsetneq \mathbb{Z}[\zeta,\zeta^{-1}]$
in~(\ref{eq:eindimen}), whereas~(\ref{eq:rafanadal}) would remain unaffected.
It will become apparent that indeed the answers to the questions above 
are in general closely related to whether $\zeta\in{S^{\ast}}$ or not. 

\subsection{Reduced form and prime denominators}
Throughout the rest of the paper
we will restrict to the non-trivial case $[\mathbb{Q}(\zeta):\mathbb{Q}]=k\geq 2$.
As indicated after Proposition~\ref{eindimfall},
we will frequently distinguish the cases $\zeta\in{S^{\ast}}$ and $\zeta\in{S\setminus{S^{\ast}}}$.

By~(\ref{eq:ungarturm}), in order to describe the elements $\alpha$ in $\mathscr{M}_{\zeta}$,
we can restrict to $\alpha\in{\mathbb{Q}(\zeta)}$. 
Clearly $\mathbb{Q}(\zeta)$ has dimension $k$ as a vector space over $\mathbb{Q}$ spanned by
 $\{1,\zeta,\zeta^{2},\ldots,\zeta^{k-1}\}$. Thus any element $a$ of $\mathscr{M}_{\zeta}$ can be written as
\[
 a=\frac{a_{0}}{b_{0}}+\frac{a_{1}}{b_{1}}\zeta+\cdots+\frac{a_{k-1}}{b_{k-1}}\zeta^{k-1}
\]
with $(a_{i},b_{i})=1$, and taking common denominators we obtain
\begin{equation} \label{eq:reducedform}
a=\frac{r_{0}+r_{1}\zeta+\cdots +r_{k-1}\zeta^{k-1}}{N}
\end{equation}
with integers $r_{j},N$ with the property $\rm{gcd}(r_{0},r_{1},\ldots,r_{k-1},N)=1$. We will refer 
to~\eqref{eq:reducedform} 
as the {\em reduced form} of an element in $\mathbb{Q}(\zeta)$ or $\mathscr{M}_{\zeta}$.
We should mention that if and only if $N=1$, the reduced element is actually 
in the ring $\mathbb{Z}[\zeta]$, which is the free $\mathbb{Z}$-submodule of $\mathcal{O}_{\mathbb{Q}(\zeta)}$
of dimension $k$ spanned by $\{1,\zeta,\zeta^{2},\ldots,\zeta^{k-1}\}$. In this context
see also Theorem~\ref{narrer}.

\begin{definition} \label{ijnt}
Let $\pi_{q}:\mathbb{Z}\mapsto \mathbb{Z}_{q}$ be the canonical reduction 
of an integer mod $q$, where we regard $\mathbb{Z}_{q}$ as a field. Denote the extension to the polynomial rings
$\mathbb{Z}[X]\mapsto \mathbb{Z}_{q}[X]$ by reducing all coefficients mod $q$
by $\pi_{q}$ as well. For a Pisot polynomial $P_{\zeta}$, let $P_{\zeta,q}$ be the reduction
of $P_{\zeta}$ over $\mathbb{Z}_{q}$, i.e., all coefficients reduced mod $q$.
\end{definition}

The following lemma is the centerpiece for the remainder of the paper.
Basically it describes all nontrivial elements with prime denominators
in reduced form and will be carried out in Theorem~\ref{kamou}. 

\begin{lemma} \label{daslemma}
Let $\zeta\in{S}$ and $P_{\zeta,q}\in{\mathbb{Z}_{q}[X]}$ as in Definition~\ref{ijnt}.
Let $P_{\zeta,q}$ factor as 
$P_{\zeta,q}=(X-\alpha_{1})^{n_{1}}(X-\alpha_{2})^{n_{2}}\cdots(X-\alpha_{g})^{n_{g}}$
over its {\upshape(}finite{\upshape)} splitting field $\mathbb{F}_{q^{h}}$,
labeled such that $q\nmid n_{j}$ for $1\leq j\leq v$ and $q\vert n_{j}$ for $v+1\leq j\leq g$
{\upshape(}with $v=0$ respectively $v=g$ if one set is empty{\upshape)}.
Further define $M(X)=1$ if $\alpha_{j}\neq 0$ for $1\leq j\leq v$
and else $M(X)=X$, and let 
\[
R_{\zeta,q}(X):=\frac{(X-\alpha_{1})(X-\alpha_{2})\cdots(X-\alpha_{v})}{M(X)}\in{\mathbb{Z}_{q}[X]}.
\]
Let $Q(X)\in{\mathbb{Z}[X]}$ arbitrary
and say $Q_{q}=\pi_{q}(Q)$ is its reduction to $\mathbb{Z}_{q}[X]$.

Then we have $\frac{Q(\zeta)}{q}\in{\mathscr{M_{\zeta}}}$ if and only if
$R_{\zeta,q}(X)\vert Q_{q}(X)$ over $\mathbb{Z}_{q}[X]$.
\end{lemma}

\begin{proof}
First note that indeed $R_{\zeta,q}\in{\mathbb{Z}_{q}[X]}$ by~(i) of Theorem~\ref{polvoll}, 
so everything is well-defined.
We first prove that the condition $R_{\zeta,q}\vert Q_{q}$ is necessary. Say
\[
a= \frac{r_{0}+r_{1}\zeta+\cdots +r_{k-1}\zeta^{k-1}}{q}
\]
is an arbitrary reduced element of $\mathscr{M}_{\zeta}$ with denominator $q$.
For $\mathscr{M}_{\zeta}$ is a $\mathbb{Z}[\zeta]$-module, all the numbers $a,\zeta a,\zeta^{2}a,\ldots$
belong to $\mathscr{M}_{\zeta}$. Since $\Vert \zeta^{n}\Vert \to 0$ and by definition of $\mathscr{M}_{\zeta}$,
we must have 
\begin{equation} \label{eq:wiesoll}
r_{0}\scp{\zeta^{n}}+r_{1}\scp{\zeta^{n+1}}+\cdots+r_{k-1}\scp{\zeta^{n+k-1}}\equiv 0 \bmod q, \qquad n\geq n_{0}.
\end{equation}
In Theorem~\ref{elemsatz} we saw $\scp{\zeta^{n}}=\sum_{j=1}^{k} x_{j}^{n}$ for sufficiently
large $n\geq n_{0}$ for $x_{1},x_{2},\ldots,x_{k}$ the roots of $P_{\zeta}$ in $\mathbb{C}$.
Let $\delta_{1},\delta_{2},\ldots,\delta_{k}$ be the roots of $P_{\zeta,q}\in{\mathbb{Z}_{q}[X]}$
over its splitting field $\mathbb{F}_{q^{h}}$, counted with multiplicity, 
such that for each $1\leq j\leq g$ there are precisely $n_{j}$ indices 
$i\in{\{1,2,\ldots,k\}}$ such that $\delta_{i}=\alpha_{j}$. 

Reducing~(\ref{eq:wiesoll}) over $\mathbb{Z}_{q}$ and denoting $\overline{a}=\pi_{q}(a)$ 
the class of an integer $a\bmod q$, we clearly have 
\begin{equation} \label{eq:nutzenm}
\overline{\scp{\zeta^{n}}}=\sum_{j=1}^{k} \delta_{j}^{n}, \qquad n\geq n_{0}
\end{equation}
Put $r_{k-1}X^{k-1}+r_{k-2}X^{k-2}+\cdots+r_{0}=:Q(X)\in{\mathbb{Z}[X]}$,
let $Q_{q}$ be the reduction of $Q$ over $\mathbb{Z}_{q}[X]$ derived by $r_{i}\mapsto \overline{r}_{i}$
and put $\widetilde{Q}_{q}(X):=X^{n_{0}}\cdot Q_{q}(X)$.
Then interpreting~(\ref{eq:wiesoll}) as equality over $\mathbb{Z}_{q}$ via 
$\scp{\zeta^{n}}\mapsto \overline{\scp{\zeta^{n}}}$
and recalling~(\ref{eq:nutzenm}) we have the identity
\[
\sum_{j=1}^{k} \delta_{j}^{n}\widetilde{Q}_{q}(\delta_{j})= 0, \qquad n\geq 0
\]
over $\mathbb{F}_{q^{h}}$. So we may take any $\mathbb{F}_{q^{h}}$-linear combinations of monomials $\delta_{j}^{n}$,
so in fact we have 
$\sum_{j=1}^{k}T(\delta_{j})\widetilde{Q}_{q}(\delta_{j})=
\sum_{j=1}^{g}n_{j}T(\alpha_{j})\widetilde{Q}_{q}(\alpha_{j})=0$ 
over $\mathbb{F}_{q^{h}}$ for all $T\in{\mathbb{F}_{q^{h}}[X]}$.  

However, by~(ii) of Theorem~\ref{polvoll} any function $\mathbb{F}_{q^{h}}^{g}\mapsto \mathbb{F}_{q^{h}}^{g}$ 
can be realized as a polynomial $T\in{\mathbb{F}_{q^{h}}[X]}$,
so actually any linear combination satisfies $\sum_{j=1}^{g} n_{j}\gamma_{j}\widetilde{Q}_{q}(\alpha_{j})=0$ 
for all $\gamma_{1},\gamma_{2},\ldots,\gamma_{g}\in{\mathbb{F}_{q^{h}}}$. 
This is easily seen to be equivalent to $n_{j}\widetilde{Q}_{q}(\alpha_{j})=0$ for all $j$,
thus for any $j$ we have $\widetilde{Q}_{q}(\alpha_{j})= 0$ or $q\vert n_{j}$.
This means that $\widetilde{Q}_{q}$ is indeed divisible by all the linear factors $(X-\alpha_{j})$
for $1\leq j\leq v$, and by definition the same applies to $Q_{q}$ where $\alpha_{j}=0$
is the only possible exception. Thus indeed $R_{\zeta,q}\vert Q_{q}$. 

To see the condition is sufficient one can basically read the proof in reverse direction.
\end{proof}

\begin{remark}
Some of the elements constructed in Lemma~\ref{daslemma} may be in $\mathcal{O}_{\mathbb{Q}(\zeta)}$,
however we will see in general we obtain ``new'' elements as well.
\end{remark}

We will discuss the existence of elements $Q(\zeta)/q^{m}\in{\mathscr{M}_{\zeta}}$ 
for prime powers $m\geq 2$ and more general $Q(\zeta)/N\in{\mathscr{M}_{\zeta}}$
for arbitrary $N$ in Theorems~\ref{neuht},\ref{okkident}. 

\begin{definition}
For a prime $q$, let $\iota_{q}:\mathbb{Z}_{q}\mapsto \mathbb{Z}$ be the inverse of $\pi_{q}$,
i.e., mapping the representation system $\{0,1,2,\ldots,q-1\}$ of $\mathbb{Z}_{q}$
identically and treating it as an element of $\mathbb{Z}$. Denote the 
extension $\iota_{q}:\mathbb{Z}_{q}[X]\mapsto \mathbb{Z}[X]$, by applying $\iota_{q}$
to all coefficients of a polynomial $P\in{\mathbb{Z}_{q}[X]}$, by $\iota_{q}$ as well. 
\end{definition}

Now we can make first assertions on nontrivial elements in $\mathscr{M}_{\zeta}$.

\begin{theorem} \label{kamou}
Let $\zeta\in{S}$ with Pisot polynomial $P_{\zeta}\in{\mathbb{Z}[X]}$.
Then exactly for those primes $q$ diving $\Delta(P_{\zeta})\cdot P_{\zeta}(0)$ there exist
elements in $\mathscr{M}_{\zeta}\setminus{\mathbb{Z}[\zeta]}$ with denominators
divisible by $q$ in reduced form~{\upshape(\ref{eq:reducedform})}.
\end{theorem}

\begin{proof}
We use the notation of Lemma \ref{daslemma}.
Note that if $q\nmid \Delta(P_{\zeta})$, then by Theorem~\ref{diskriminante} we have
$n_{j}=1$ for all $1\leq j\leq g$. Similarly, if $q\nmid P_{\zeta}(0)$ we have
$M(X)=1$. In view of this, in case of $q\nmid \Delta(P_{\zeta})P_{\zeta}(0)$
we infer $R_{\zeta,q}= P_{\zeta,q}$, and $\deg(R_{\zeta,q})=\deg(P_{\zeta,q})=\deg(P_{\zeta})=k$. 
Clearly $\deg(P)\geq \deg(\pi_{q}(P))$ holds for any polynomial
$P\in{\mathbb{Z}[X]}$. 
Thus, any polynomial $Q$ with $Q(\zeta)/q\in{\mathscr{M}_{\zeta}}$
and by Lemma~\ref{daslemma} $R_{\zeta,q}\vert Q_{q}:=\pi_{q}(Q)$,
would have degree 
\[
\deg(Q)\geq \deg(Q_{q})\geq \deg(R_{\zeta,q})=k.
\]
Hence $Q(\zeta)/q$ would not be in
reduced form~(\ref{eq:reducedform}), so indeed there exist no reduced elements with denominator $q$.

Else if $q\vert \Delta(P_{\zeta})$ then at least one linear factor $(X-\alpha_{j})$ 
of $P_{\zeta,q}$ has power $n_{j}\geq 2$ in $P_{\zeta,q}$ by Theorem~\ref{diskriminante}. 
So regardless of whether $1\leq j\leq v$ or $v+1\leq j\leq g$,
the polynomial $R_{\zeta,q}(X)$ has at least one factor $(X-\alpha_{j})$ less than $P_{\zeta,q}$.
Similarly, if $q\vert P_{\zeta}(0)$ then $X\vert P_{\zeta,q}$, so either $M(X)=X$ or 
$X$ is no factor in $R_{\zeta,q}$ anyway, hence in both cases 
\[
R_{\zeta,q}(X)\vert \frac{P_{\zeta,q}(X)}{X}.
\]
In any case, putting $Q_{q}(X):=R_{\zeta,q}(X)$ and treating $Q:=\iota_{q}(Q_{q})$ as an element of $\mathbb{Z}[X]$,
the above shows $Q$ has degree strictly less than $k$. Moreover it is not identically $0$, 
since it is monic by construction.   
Thus $Q(\zeta)/q$ is a nontrivial element in reduced form, which is in $\mathscr{M_{\zeta}}$ 
due to Lemma~\ref{daslemma}. We have proved both directions of the assertion.
\end{proof}

From Theorem~\ref{kamou} we deduce some corollaries.

\begin{corollary}  \label{olympspiele}
Let $\zeta\in{S}$ and put $m:=\rm{rad}(P_{\zeta}(0)\Delta(P_{\zeta}))$. We have 
\[
\mathscr{M}_{\zeta}\subset \mathbb{Z}[\zeta,1/m].
\]
In particular, $\mathscr{M}_{\zeta}$ is contained in a finitely generated ring over $\mathbb{Z}$.
\end{corollary}

\begin{proof}
Say $m=p_{1}p_{2}\cdots p_{u}$ and let $a$ be an arbitrary element of $\mathscr{M}_{\zeta}$. 
Theorem~\ref{kamou} implies that all elements of $\mathscr{M}_{\zeta}$ in reduced form 
have denominators consisting only of prime factors $p_{1},p_{2},\ldots, p_{u}$.
Thus $a$ is of the form
\[
a=\frac{r_{0}+r_{1}\zeta+\cdots +r_{k-1}\zeta^{k-1}}{p_{1}^{\alpha_{1}}p_{2}^{\alpha_{2}}\cdots p_{u}^{\alpha_{u}}}.
\]
Let $b:=r_{0}+r_{1}\zeta+\cdots +r_{k-1}\zeta^{k-1}\in{\mathbb{Z}[\zeta]}$ 
and $\alpha:=\max_{1\leq j\leq u} \alpha_{j}$. Then $a=Mb/m^{\alpha}$ for the integer 
$M:=\prod_{j=1}^{u} p_{j}^{\alpha-\alpha_{j}}$, so clearly $a\in{\mathbb{Z}[\zeta,1/m]}$.
\end{proof}

\begin{remark}
By Proposition~\ref{notring}, in general $\mathscr{M}_{\zeta}$ itself is no ring.
\end{remark}

The next corollary is not very surprising and more for sake of completeness.

\begin{corollary}
For any $\zeta\in{S}$, we have $\mathscr{M}_{\zeta}\subsetneq \mathbb{Q}(\zeta)$.
In particular, $\mathscr{M}_{\zeta}$ is never a field.
\end{corollary}

\begin{proof}
For any prime $q>\Delta(P_{\zeta})P_{\zeta}(0)$ we have $1/q\notin{\mathscr{M}_{\zeta}}$
by Theorem~\ref{kamou}, but clearly $1/q\in{\mathbb{Q}}\subset{\mathbb{Q}(\zeta)}$.
If $\mathscr{M}_{\zeta}$ was a field then by $\mathcal{O}_{\mathbb{Q}(\zeta)}\subset \mathscr{M}_{\zeta}$ 
we would have that the quotient field of $\mathcal{O}_{\mathbb{Q}(\zeta)}$ would be contained in $\mathscr{M}_{\zeta}$. 
This quotient field coincides with $\mathbb{Q}(\zeta)$ by~(i) of Theorem~\ref{dastheorem}, 
so we just contradicted this inclusion.
\end{proof}

\begin{corollary}  \label{dazwischen}
There exist $\zeta\in{S}$ such that $\mathscr{M}_{\zeta}$ is a $\mathcal{O}_{\mathbb{Q}(\zeta)}[\zeta^{-1}]$-module
but not a ring. 
\end{corollary}

\begin{proof}  
The Pisot root $\zeta= \sqrt{2}+1$ of $P_{\zeta}(X)=X^{2}-2X-1$ is a unit satisfying
$\Delta(P_{\zeta})=8$, so all reduced elements in $\mathscr{M}_{\zeta}$ have denominators
that are powers of $2$ by Theorem~\ref{kamou}. One checks the period of the recurrence relation from
Theorem~\ref{elemsatz} mod $8$ is $\overline{2,2,6,6}$.
It is not hard to conclude $1/2\in{\mathscr{M}_{\zeta}}, 
(1+\zeta)/4=(2+\sqrt{2})/4\in{\mathscr{M}_{\zeta}}$ and further by $\mathbb{Z}$-module
property that $a/2+b/(2\sqrt{2})\in{\mathscr{M}_{\zeta}}$ for all integers $a,b$ and conversely
$\mathscr{M}_{\zeta}$ consists only of such elements.
Since $(1/2)^{2}=1/4\notin{\mathscr{M}_{\zeta}}$ 
it is clearly no ring. On the other hand, by Theorem~\ref{kopernik} we have 
$\mathcal{O}_{\mathbb{Q}(\zeta)}=\mathcal{O}_{\mathbb{Q}(\sqrt{2})}=\mathbb{Z}+\sqrt{2}\mathbb{Z}$,
and since elements of the form $a/2+b/(2\sqrt{2})$ are closed under multiplication
with such elements, $\mathscr{M}_{\zeta}$ is a $\mathcal{O}_{\mathbb{Q}(\zeta)}$-module.
Finally, $\zeta$ being a unit 
we have that $\mathcal{O}_{\mathbb{Q}(\zeta)}=\mathcal{O}_{\mathbb{Q}(\zeta)}[\zeta^{-1}]$.
\end{proof}

\subsection{The module structure of $\mathscr{M}_{\zeta}$}
The main result of this section is Corollary~\ref{ontothetop}. It will be based on
the extensions Theorems~\ref{neuht}, \ref{okkident} of Theorem~\ref{kamou} for elements 
with denominators not necessarily prime in reduced form. 

\begin{definition}
We will call $\zeta\in{S}$ {\em regular at a prime $q$} if
$q\nmid P_{\zeta}(0)$ in $\mathbb{Z}$, i.e., $q$ does not divide the constant coefficient of $P_{\zeta}$.
Else we will call $\zeta$ {\em singular at $q$}.
\end{definition}

Obviously $\zeta\in{S}$ is regular at all prime numbers if and only if $\zeta\in{S^{\ast}}$.

\begin{theorem} \label{neuht}             
Let $\zeta\in{S}$ be singular at a prime $q$. 
Denote by $m\geq 1$ the multiplicity of $0$ as a root of $P_{\zeta,q}$ over $\mathbb{Z}_{q}$. 
Then for any $l\geq 1$ the element $\zeta^{-lm}\in{\mathbb{Z}[\zeta,\zeta^{-1}]}$ is 
an element with denominator divisible by $q^{l}$ in reduced form. 

In particular, for any positive integer $l$ there exist reduced elements in 
$\mathbb{Z}[\zeta,\zeta^{-1}]\subset \mathscr{M}_{\zeta}$ with denominator $q^{l}$.
\end{theorem}

\begin{proof}
Write $P_{\zeta}(X)=X^{k}+a_{k-1}X^{k-1}+\cdots+a_{1}X+a_{0}$.
Putting $a_{k}=1$, the value $m$ in the theorem coincides with
the minimum index $j\in\{1,2,\ldots,k\}$ with $q\nmid a_{j}$. 

Clearly 
\[
\zeta^{-1}=-\frac{\zeta^{k-1}+a_{k-1}\zeta^{k-2}+\cdots+a_{1}}{a_{0}}
\]
is already reduced ($a_{0}\neq 0$ because $P_{\zeta}$ is irreducible). 
Putting $Q(X):=-X^{k-1}-a_{k-1}X^{k-2}-\cdots-a_{1}$ we have
$\zeta^{-l}=(1/a_{0}^{l})Q(\zeta)^{l}$. If $q$ does not divide all coefficients
of $Q^{l}(X) \bmod P_{\zeta}(X)$ the element $\zeta^{-l}$
has denominator divisible by $q^{l}$ in reduced form and we are done.
Thus putting $Q_{q}=\pi_{q}(Q),P_{\zeta,q}=\pi_{q}(P_{\zeta})$ 
the reductions of $Q,P_{\zeta}$ over $\mathbb{Z}_{q}[X]$, we are left
to prove $Q_{q}^{l} \bmod P_{\zeta,q}\neq 0$ as a polynomial over $\mathbb{Z}_{q}[X]$, 
or equivalently $P_{\zeta,q}\nmid Q_{q}^{l}$, for all $l\geq 1$. 
By construction $P_{\zeta,q}=-X^{m}\cdot Q_{q}$, so assuming $P_{\zeta,q}\vert Q_{q}^{l}$ would
yield $X\vert Q_{q}$ and hence $X^{m+1}\vert P_{\zeta,q}$, a contradiction to our assumption $q\nmid a_{m}$.
\end{proof}
Now we turn to the regular primes $q$ of given $\zeta\in{S}$.

\begin{definition}
For an integer $N$, let $\nu_{p}(N)$ be the multiplicity of the prime $p$ in $N$.
\end{definition}

\begin{theorem} \label{okkident}
If $\zeta\in{S}$ is regular at a prime $q$, the multiplicity of $q$ in the denominator 
of any element in $\mathscr{M}_{\zeta}$ is at most $\nu_{q}(\Delta(P_{\zeta}))$.
Thus if $\zeta\in{S^{\ast}}$, in fact every denominator of an element 
of $\mathscr{M}_{\zeta}$ in reduced form divides $\Delta(P_{\zeta})$.
\end{theorem}

\begin{proof}
Assume $\zeta$ is regular at $q$, so $q\nmid P_{\zeta}(0)$. We have to show
that a reduced element with denominator $q^{m}$ in $\mathscr{M}_{\zeta}$ 
implies $q^{m}\vert \Delta(P_{\zeta})$. Say
\[
a= \frac{r_{0}+r_{1}\zeta+\cdots +r_{k-1}\zeta^{k-1}}{q^{m}}
\]
is an arbitrary reduced element of $\mathscr{M}_{\zeta}$ with denominator $q^{m}$ 
and assume $q^{m}\nmid \Delta(P_{\zeta})$. Further put $Q(X):=r_{k-1}X^{k-1}+r_{k-2}X^{k-2}+\cdots+r_{0}$ 
and $\underline{r}=(r_{0},r_{1},\ldots,r_{k-1})$. Further let
$x_{1}=\zeta,x_{2},\ldots,x_{k}\in{\mathcal{O}_{K}}$ be the roots of $P_{\zeta}$ in the complex plane
and $K=\mathbb{Q}(x_{1},x_{2},\ldots,x_{k})$ be the splitting field of $P_{\zeta}$.
Proceeding as in Lemma~\ref{daslemma}, by~(\ref{eq:wiesoll}) and
$\scp{\zeta^{n}}=\sum_{j=1}^{k} x_{j}^{n}$, we infer
\[
\sum_{j=1}^{k} x_{j}^{n}Q(x_{j}) \equiv 0 \bmod q^{m}, \qquad n\geq n_{0}.
\]
Hence also $\sum_{j=1}^{k} P(x_{j})x_{j}^{n_{0}}Q(x_{j})\equiv 0 \bmod q^{m}$ for any polynomial $P\in{\mathbb{Z}[X]}$.
This can be written as
\[
\left( \begin{array}{cccccccccc}
P(x_{1}) \quad P(x_{2}) \quad \cdots \quad P(x_{k})
\end{array} \right) \cdot
\left( \begin{array}{cccccccccc}
x_{1}^{n} & x_{1}^{n+1} & \ldots & x_{1}^{n+k-1}      \\
x_{2}^{n} & x_{2}^{n+1} & \ldots & x_{2}^{n+k-1}\\
\vdots & \vdots  & \vdots & \vdots \\
x_{k}^{n} & x_{k}^{n+1} & \ldots & x_{k}^{n+k-1}  \\
\end{array} \right) \cdot
\left( \begin{array}{cccccccccc}
r_{0}     \\
r_{1}\\
\vdots \\
r_{k-1} \\
\end{array} \right)= q^{m}N 
\]
for some $N\in{\mathbb{Z}}$ depending on $P,n$. Let $B_{n}$ be the system matrix.
By the free choice of $P$ we may put $P_{j}(X)=X^{j}$ for $0\leq j\leq k-1$ to obtain
\begin{equation} \label{eq:system}
\left( \begin{array}{cccccccccc}
1 & 1 & \ldots & 1      \\
x_{1} & x_{2} & \ldots & x_{k}\\
\vdots & \vdots  & \vdots & \vdots \\
x_{1}^{k-1} & x_{2}^{k-1} & \ldots & x_{k}^{k-1}  \\
\end{array} \right) \cdot
\left( \begin{array}{cccccccccc}
x_{1}^{n} & x_{1}^{n+1} & \ldots & x_{1}^{n+k-1}      \\
x_{2}^{n} & x_{2}^{n+1} & \ldots & x_{2}^{n+k-1}\\
\vdots & \vdots  & \vdots & \vdots \\
x_{k}^{n} & x_{k}^{n+1} & \ldots & x_{k}^{n+k-1}  \\
\end{array} \right) \cdot
\left( \begin{array}{cccccccccc}
r_{0}     \\
r_{1}\\
\vdots \\
r_{k-1} \\
\end{array} \right)= q^{m}\left( \begin{array}{cccccccccc}
N_{0,n}     \\
N_{1,n}\\
\vdots \\
N_{k-1,n} \\
\end{array} \right). 
\end{equation}
Equivalently $C_{n}\underline{r}=q^{m}\underline{N}_{n}$ with $C_{n}:=B_{0}^{t}B_{n}$
and $\underline{N}_{n}:=(N_{0,n},N_{1,n},\ldots,N_{k-1,n})\in{\mathbb{Z}^{k}}$, where 
$.^{t}$ denotes the transpose matrix. Since $P_{\zeta}$ is monic we can 
interpret~(\ref{eq:system}) over $\mathcal{O}_{K}$ as well.
By Cramer's rule we have the equality 
\begin{equation} \label{eq:wienerwald}
r_{j}=\det (C_{n,j})\det (C_{n})^{-1}, \qquad 1\leq j\leq k,
\end{equation}
where $C_{n,j}$ is the matrix arising from $C_{n}$ by replacing the $j$-th 
column by $q^{m}\underline{N}_{n}$. Interpreting~(\ref{eq:wienerwald}) over $\mathcal{O}_{K}$
and transitioning to generated principal ideals denoted by $(P)=P\mathcal{O}_{K}$, 
we obtain 
\begin{equation} \label{eq:winerwald}
(r_{j})(\det(C_{n}))= (\det(C_{n,j})), \qquad 1\leq j\leq k.
\end{equation}
Say $(q)$ factors as $(q)=Q_{1}^{n_{1}}Q_{2}^{n_{2}}\cdots Q_{u}^{n_{u}}$, which is unique 
by~(ii) of Theorem~\ref{dastheorem}.
We will show that the multiplicity of the $Q_{i}$ cannot be equal in~(\ref{eq:winerwald}) for all $1\leq j\leq k$.

Using the identity from Theorem~\ref{discoduck} and the multilinearity of the determinant, we obtain
$\det(C_{n})=\det(B_{0})\det(B_{n})=(x_{1}x_{2}\cdots x_{k})^{n}\det(B_{0})^{2}
=(-1)^{kn}P_{\zeta}(0)^{n}\Delta(P_{\zeta})$. 
By regularity of $\zeta$ at $q$ and our assumption $q^{m}\nmid \Delta(P_{\zeta})$ we infer 
$\det(C_{n})=q^{s}t_{n}$ for some $s=\nu_{q}(\Delta(P_{\zeta}))<m$ and 
$t_{n}\in{\mathbb{Z}}$ with $q\nmid t_{n}$ in $\mathbb{Z}$. So making transition to principal
ideals we have $(\det(C_{n}))=(q)^{s}(t_{n})$. By~(iii) of Theorem~\ref{dastheorem}, the ideal $(t_{n})$
has no factor $Q_{i}$ in its prime ideal factorization. By our assumption that $a$ is reduced,
there exists at least one index $1\leq j_{0}\leq k$ such that $q\nmid r_{j_{0}}$ in $\mathbb{Z}$.
So again by~(iii) of Theorem~\ref{dastheorem}, the ideal $(r_{j_{0}})$ has no factor $Q_{i}$ in its prime ideal decomposition.  
Combining these properties, the left hand side of~(\ref{eq:winerwald})
contains the prime ideal $Q_{i}$ with multiplicity exactly $sn_{i}$ for $1\leq i\leq u$ and $j=j_{0}$.

On the other hand, expanding the determinant of $C_{n,j}$,
we can write $\det (C_{n,j})= q^{m}\eta_{n,j}$ for some $\eta_{n,j}\in{\mathcal{O}_{K}}$.
For the generated prime ideals we infer $(\det (C_{n,j}))= (q)^{m}(\eta_{n,j})$, so 
the multiplicity of $Q_{i}$ is at least $mn_{i}$. Thus equality in~(\ref{eq:winerwald}) for $j=j_{0}$
contradicts our hypothesis $s<m$. 
Thus the assumption $q^{m}\nmid \Delta(P_{\zeta})$ cannot hold and we are done.  
\end{proof}

\begin{remark}
In general, there is no equality of exponents in Theorem~\ref{okkident}.
For instance, for the Pisot polynomial $P_{\zeta}(X)=X^{2}-2X-1$ of 
$\zeta= \sqrt{2}+1$ we have $\Delta(P_{\zeta})=8$.
However, we saw in Corollary~\ref{dazwischen},
that there is no reduced element in $\mathscr{M}_{\zeta}$
with denominator $8$, though $\frac{1+\zeta}{4}\in{\mathscr{M}_{\zeta}}$. 
A sufficient condition for equality is $\nu_{q}(\Delta(P_{\zeta}))=1$ by Theorem~\ref{kamou}, though. 
\end{remark}
Theorem~\ref{okkident} yields an improvement of Corollary~\ref{olympspiele}.
\begin{corollary}
Let $\zeta\in{S}$. With $\presuc{.}$ as in Definition \ref{defff}, we have the inclusion
\[
\mathscr{M}_{\zeta}\subset \mathbb{Z}[\zeta,1/P_{\zeta}(0)]\presuc{1/\Delta(P_{\zeta})}.
\]
\end{corollary}
\begin{proof}
By Theorem~\ref{kamou} any element $a\in{\mathscr{M}_{\zeta}}$ can be written
$a=m/N$ with $m\in{\mathbb{Z}[\zeta]},N=P_{\zeta}(0)^{r}\Delta(P_{\zeta})^{s}$ with non-negative
integers $r,s$, and by Theorem~\ref{okkident} we may additionally assume $s\leq 1$. This immediately
implies the corollary. 
\end{proof}
An easy observation shows all denominators that occur among reduced elements 
of $\mathscr{M}_{\zeta}$ are traceable to the prime power case of Theorems~\ref{neuht}, \ref{okkident}.
\begin{proposition}
Let $\zeta\in{S}$ and $\alpha_{1}=\frac{P_{1}(\zeta)}{N_{1}},\alpha_{2}=\frac{P_{2}(\zeta)}{N_{2}},
\ldots,\alpha_{l}=\frac{P_{l}(\zeta)}{N_{l}}$ be reduced elements of $\mathscr{M}_{\zeta}$
with denominators $N_{j}=p_{j}^{\beta_{j}}$ for $p_{1},p_{2},\ldots,p_{l}$ pairwise distinct
primes. Then the element $\sum_{j=1}^{l} \alpha_{j}\in{\mathscr{M}_{\zeta}}$ has denominator 
$\prod_{j=1}^{l} N_{j}=\prod_{j=1}^{l} p_{j}^{\beta_{j}}$ in reduced form.
\end{proposition}  

\begin{proof}
Clearly $\sum_{j=1}^{l} \alpha_{j}\in{\mathscr{M}_{\zeta}}$ by the 
$\mathbb{Z}$-module property. On the other hand, if we write
\begin{equation} \label{eq:rumpel}
\sum_{j=1}^{l} \alpha_{j}= 
\frac{A_{1}(\zeta)+A_{2}(\zeta)+\cdots +A_{l}(\zeta)}
{N_{1}N_{2}\cdots N_{l}}, \quad A_{j}(X):= \prod_{i\neq j} P_{i}(X)N_{i},
\end{equation}
for any $1\leq j\leq l$, only $A_{j}(X)$ does not 
have all its coefficients divisible by $p_{j}$. 
Hence no $p_{j}$ divides the nominator of the right hand side
of~(\ref{eq:rumpel}) and thus it is already reduced.  
\end{proof}

Theorem~\ref{strucsatz} immediately gives some information 
on the module shape of $\mathscr{M}_{\zeta}$.

\begin{proposition} \label{soneprop}
If $\mathscr{M}_{\zeta}$ is free as a $\mathbb{Z}$-module, then its dimension equals $k$, so 
$\mathscr{M}_{\zeta}\cong \mathbb{Z}^{k}$ as a $\mathbb{Z}$-module.
If $\mathscr{M}_{\zeta}$ is free as a $\mathbb{Z}[\zeta]$-module 
{\upshape(}or $\mathbb{Z}[\zeta,\zeta^{-1}]$-module{\upshape)}, then its dimension equals $1$, 
so $\mathscr{M}_{\zeta}\cong \mathbb{Z}[\zeta]$ as a $\mathbb{Z}[\zeta]$-module
{\upshape(}or $\mathscr{M}_{\zeta}\cong \mathbb{Z}[\zeta,\zeta^{-1}]$ as 
a $\mathbb{Z}[\zeta,\zeta^{-1}]$-module{\upshape)}.
\end{proposition}

\begin{proof}
Suppose $\mathscr{M}_{\zeta}$ is a free $\mathbb{Z}$-module. 
Since $\mathscr{M}_{\zeta}\subset \mathbb{Q}(\zeta)\cong \mathbb{Q}^{k}$, any $k+1$ elements
are linearly dependent over $\mathbb{Z}$, so the dimension of $\mathscr{M}_{\zeta}$ is at most $k$.
By Theorem~\ref{strucsatz} it is of the form $\mathscr{M}_{\zeta}\cong \mathbb{Z}^{l}$ for some $1\leq l\leq k$.
Finally we cannot have $l<k$, since as $\mathscr{M}_{\zeta}$ contains $\mathbb{Z}[\zeta]$, which is
the free $\mathbb{Z}$-submodule of dimension $k$
spanned by $\{1,\zeta,\zeta^{2},\ldots,\zeta^{k-1}\}$, this would contradict Theorem~\ref{strucsatz}. 

For the assertion on $\mathscr{M}_{\zeta}$ as a $\mathbb{Z}[\zeta]$-module (or $\mathbb{Z}[\zeta,\zeta^{-1}]$-module),
it suffices to prove that any $2$ distinct elements of $\mathbb{Q}(\zeta)$
are linearly dependent over $\mathbb{Z}[\zeta]$ (or $\mathbb{Z}[\zeta,\zeta^{-1}]$). 
This in easily inferred by writing them
in reduced form~(\ref{eq:reducedform}). If $a=A_{1}/N_{1},b=A_{2}/N_{2}$ 
with $A_{j}\in{\mathbb{Z}[\zeta]}, N_{j}\in{\mathbb{Z}}$, then $N_{1}A_{2}\cdot a-N_{2}A_{1}\cdot b=0$ 
is a nontrivial linear combination over $\mathbb{Z}[\zeta]$.
\end{proof}

We combine Proposition~\ref{soneprop}
with Theorem~\ref{okkident} to describe the algebraic shape of $\mathscr{M}_{\zeta}$.
Corollary~\ref{ontothetop} shows that, although we will show in Corollary~\ref{mrmr}
we always have $\mathscr{M}_{\zeta}\neq\mathbb{Z}[\zeta]$, 
in many cases there are module isomorphisms $\mathscr{M}_{\zeta}\cong\mathbb{Z}^{k}$ and
$\mathscr{M}_{\zeta}\cong \mathbb{Z}[\zeta]$, with $\mathscr{M}_{\zeta}$ viewed 
as a $\mathbb{Z}$-module and a $\mathbb{Z}[\zeta]$-module, respectively.

\begin{corollary}  \label{ontothetop}
Let $\zeta\in{S}$. The set $\mathscr{M}_{\zeta}$ is finitely generated
as a $\mathbb{Z}$-module or a $\mathbb{Z}[\zeta]$-module 
if and only if $\zeta\in{S^{\ast}}$.

If $\mathscr{M}_{\zeta}$ is free as a $\mathbb{Z}$-module or
$\mathbb{Z}[\zeta]$-module, then $\zeta\in{S^{\ast}}$.
Conversely, if $\zeta\in{S^{\ast}}$ is of degree $k$, then $\mathscr{M}_{\zeta}\cong \mathbb{Z}^{k}$ 
is free as a $\mathbb{Z}$-module, and a sufficient additional condition for $\mathscr{M}_{\zeta}$
to be free as a $\mathbb{Z}[\zeta]$-module is that 
$\mathbb{Z}[\zeta]$ is a principal ideal domain, and in this case
$\mathscr{M}_{\zeta}\cong \mathbb{Z}[\zeta]$ is free of dimension one as a $\mathbb{Z}[\zeta]$-module.
\end{corollary}

\begin{proof}                                                            
First assume $\zeta\in{S^{\ast}}$. 
According to Theorem~\ref{kamou} there are only finitely
many primes $q_{1},q_{2},\ldots, q_{u}$ that divide the denominator of an element in $\mathscr{M}_{\zeta}$
in reduced form~(\ref{eq:reducedform}). 
Together with Theorem~\ref{okkident} it follows there at
most $\prod_{j=1}^{u} (\nu_{q_{j}}+1)<\infty$ denominators that can appear in the 
reduced form of an element of $\mathscr{M}_{\zeta}$. However, for any such element of 
$\mathscr{M}_{\zeta}$ with denominator $N$ in reduced form, by subtracting some element
of $\mathbb{Z}[\zeta]$ an element that written as in~(\ref{eq:reducedform}) 
has all $r_{i}\in{\{0,1,2,\ldots,N-1\}}$, can be constructed. 
The set of all such numbers, say $\Omega$, has at most $N^{k}$ elements, so it is finite.
Thus the finite set $\Omega\cup\{1,\zeta,\zeta^{2},\ldots,\zeta^{k-1}\}$
generates $\mathscr{M}_{\zeta}$ as a $\mathbb{Z}$-module. 
The isomorphism $\mathscr{M}_{\zeta}\cong \mathbb{Z}^{k}$ follows from Theorem~\ref{strucsatz}
and Proposition~\ref{soneprop}, and provided that $\mathbb{Z}[\zeta]$ is a principal ideal domain
similarly $\mathscr{M}_{\zeta}\cong \mathbb{Z}[\zeta]$.

Now suppose $\zeta\in{S\setminus{S^{\ast}}}$. Then $\zeta$ is singular at some prime $q$. 
Assume $\mathscr{M}_{\zeta}$ is finitely generated as a $\mathbb{Z}[\zeta]$-module.
For any fixed finite generating system $\Psi:=\{\varsigma_{1},\varsigma_{2},\ldots,\varsigma_{l}\}$
with denominators $N_{1},N_{2},\ldots,N_{l}$ in reduced form
put $\nu:=\max_{1\leq j\leq l}\nu_{q}(N_{j})$. It is not hard to see that
$\mathbb{Z}[\zeta]$-linear combinations of arbitrary elements in reduced form with denominators $M_{1},M_{2}$ have
denominator dividing $\rm{lcm}(M_{1},M_{2})$ in reduced form.
Thus any reduced element in the $\mathbb{Z}[\zeta]$-span of $\Psi$ with denominator $N$ has
$\nu_{q}(N)\leq \nu$ as well. Since this $\mathbb{Z}[\zeta]$-span was assumed to be entirely 
$\mathscr{M}_{\zeta}$, this contradicts Theorem~\ref{neuht}. Hence $\mathscr{M}_{\zeta}$
is not finitely generated as a $\mathbb{Z}[\zeta]$-module, let alone as a $\mathbb{Z}$-module.
By Proposition~\ref{soneprop} both modules cannot be free either. 
\end{proof}                                                   

\begin{remark}
The question if $\mathbb{Z}[\zeta]$ a principal ideal domain is closely related to the class number problem
for number fields. In case of $\mathcal{O}_{\mathbb{Q}(\zeta)}=\mathbb{Z}[\zeta]$, which is in particular
true if the discriminant $\Delta(P_{\zeta})$ is square-free, it is equivalent to the class number of 
$\mathbb{Q}(\zeta)$ being $1$.
\end{remark}

In case of $\zeta\in{S^{\ast}}$ obviosuly $\mathbb{Z}[\zeta]=\mathbb{Z}[\zeta,\zeta^{-1}]$, so in
the last assertion of Corollary~\ref{ontothetop} we can replace $\mathbb{Z}[\zeta]$-module 
by $\mathbb{Z}[\zeta,\zeta^{-1}]$-module throughout. Concerning the $\mathbb{Z}[\zeta,\zeta^{-1}]$-module
structure of $\mathscr{M}_{\zeta}$ the following property is reasonable.

\begin{conjecture}
For $\zeta\in{S}$, the set $\mathscr{M}_{\zeta}$ is finitely generated as a $\mathbb{Z}[\zeta,\zeta^{-1}]$-module.
More generally,
\begin{equation} \label{eq:fussball}
\mathscr{M}_{\zeta}\subset \mathbb{Z}[\zeta,\zeta^{-1}]\presuc{1/\Delta(P_{\zeta})}
\end{equation}
with $\presuc{.}$ again as in Definition \ref{defff}. ($\mathbb{Z}[\zeta,\zeta^{-1}]$ is Noetherian,
as it is a finitely generated algebra over the Noetherian ring $\mathbb{Z}$,
and Noetherian rings can be characterized precisely as those rings $R$ for which sub-modules of 
finitely generated modules over $R$ are finitely generated, 
see Proposition~3 Chapter~6 in \cite{lang}. Thus~(\ref{eq:fussball}) indeed 
implies that $\mathscr{M}_{\zeta}$ is finitely generated as a 
$\mathbb{Z}[\zeta,\zeta^{-1}]$-module.)
\end{conjecture}

\subsection{The inclusions in (\ref{eq:unnuetzlich})}
We want to take the investigation of the Pisot numbers from a given number field $K$ a little further.
In Section~\ref{eindimensionul} we derived some properties of the intersection and union
of $\mathscr{M}_{\zeta}$ among $\zeta\in{S_{K}}$ and $\zeta\in{S\cap K}$. We investigate this in more detail now.
Proposition~\ref{eindimfall} showed $\chi:=\zeta\longmapsto \mathscr{M}_{\zeta}$ 
from the set $S$ to $2^{\mathbb{R}}$ is not one-to-one, as already its restriction
to $S_{\mathbb{Q}}=\{2,3,\ldots\}$ is not.
Corollary~\ref{ontothetop} implies a reverse statement.

\begin{corollary} \label{brauchijetz}
If $\zeta_{1}\in{S^{\ast}}$ and $\zeta_{2}\in{S\setminus{S^{\ast}}}$, 
then $\mathscr{M}_{\zeta_{1}}\neq \mathscr{M}_{\zeta_{2}}$. In particular, in any
real number field $K$ the restricted map $\chi_{K}: \zeta\longmapsto \mathscr{M}_{\zeta}$ from $S_{K}$
to $2^{\mathbb{R}}$ {\upshape(}or $2^{K}${\upshape)} is not constant. 
\end{corollary}

\begin{proof}
The first assertion follows immediately from Corollary~\ref{ontothetop}, for the second one
it suffices to find both $\zeta\in{S_{K}^{\ast}}$ and $\zeta\in{S_{K}\setminus{S_{K}^{\ast}}}$, 
which is guaranteed by Theorem~\ref{reellfeld}.
\end{proof}

We will utilize the following identity of fields in Theorem~\ref{unglkette}.  

\begin{proposition} \label{powers}
Let $\zeta\in{S}$ and integers $N\neq 0$ and $s\geq 1$ such that $N\zeta^{s}\in{S}$ too. 
Then $\mathbb{Q}(N\zeta^{s})=\mathbb{Q}(\zeta)$.
\end{proposition}

\begin{proof}
Clearly $\mathbb{Q}(N\zeta^{s})\subset \mathbb{Q}(\zeta)$. Say $\zeta$ has degree $k$ 
and conjugates $x_{1}=\zeta,x_{2},\ldots,x_{k}$.
The polynomial $R(X)=\prod_{1\leq j\leq k}(X-Nx_{j}^{s})\in{\mathbb{Z}[X]}$ is monic of degree $k$ 
with root $N\zeta^{s}$ and non-vanishing constant coefficient $r_{0}$, and
has integral coefficients by Theorem~\ref{elemsatz}. Suppose
$\mathbb{Q}(N\zeta^{s})\subsetneq \mathbb{Q}(\zeta)$.
Then $R(X)=P(X)Q(X)$ with $P$ the minimal polynomial of $N\zeta^{s}$ and 
$Q$ of degree at least $1$. Note $P,Q$ are monic as $R$ is. Hence
Vieta Theorem~\ref{vieta} for the constant coefficients $p_{0},q_{0}$ of $P,Q$ and
$1\leq \vert r_{0}\vert=\vert p_{0}\vert\vert q_{0}\vert$ imply both $P,Q$ must have at least 
one root greater than one. Hence $R$ has at least two roots greater than one, contradicting $N\zeta^{s}\in{S}$.
\end{proof}

It will be convenient to write vectors $\underline{x}$ as line vectors and denote by
$\underline{x}^{t}$ the transpose column vector. 
Furthermore $\underline{e}_{j}=(0,0,\ldots,0,1,0,\ldots,0)$ are the canonical basis vectors.

\begin{theorem} \label{unglkette}
For all real number fields $K\neq \mathbb{Q}$ we have the inclusions
\begin{equation} \label{eq:unnuetz}
\bigcap_{\zeta\in{S\cap K}}\mathscr{M}_{\zeta}= \mathbb{Z}\subsetneq
\mathcal{O}_{K}\subset \bigcap_{\zeta\in{S_{K}}}\mathscr{M}_{\zeta}
\subsetneq \bigcup_{\zeta\in{S_{K}}} \mathscr{M}_{\zeta}=
\bigcup_{\zeta\in{S\cap K}} \mathscr{M}_{\zeta}= K.
\end{equation}
\end{theorem}

\begin{proof}
We proceed from left to right.
The first equality is due to Proposition~\ref{eindimfall}, the following proper inclusion is obvious.
By the definition of $S_{K}$ and~(\ref{eq:ungarturm}) we easily deduce the remaining inclusions. 
The forth inclusion is proper for all $K$ due to Corollary~\ref{brauchijetz}. 

It remains to directly prove $\bigcup_{\zeta\in{S_{K}}} \mathscr{M}_{\zeta}=K$. By~(i) of Theorem~\ref{dastheorem}
it suffices to show that for any $N\in{\mathbb{Z}\setminus{\{0\}}}$ the inclusion
$\{z/N: z\in{\mathcal{O}_{K}}\}=:\mathcal{O}_{K}/N\subset \mathscr{M}_{\zeta}$ 
holds.
Consider $N$ fixed and choose any $\zeta\in{S_{K}}$ of the form $\zeta=N\zeta_{0}^{L}$ with 
another $\zeta_{0}\in{S_{K}}$ and an integer $L\geq 2$. We prove such $\zeta$ exists.
Take any $\zeta_{0}\in{S_{K}}$, which exists by Theorem~\ref{reellfeld}.
If $x_{2},x_{3},\ldots,x_{k}$ are the conjugates of $x_{1}=\zeta_{0}$,
the conjugates of $\zeta=N\zeta_{0}^{L}$ are $Nx_{2}^{L},Nx_{3}^{L},\ldots,Nx_{k}^{L}$ 
and they indeed tend to $0$ as $L\to\infty$. This means for sufficiently large $L$ the
moduli of the conjugates of $\zeta$ are indeed smaller than $1$, so $\zeta\in{S\cap K}$. 
Due to Proposition~\ref{powers}, $\zeta$ actually generates $K$.
We claim for any such $\zeta\in{S_{K}}$ we have $\mathcal{O}_{K}/N\subset{\mathscr{M}_{\zeta}}$.

First we prove that $\{z/N: z\in{\mathbb{Z}[\zeta]},N\in{\mathbb{Z}}\}=:\mathbb{Z}[\zeta]/N\subset \mathscr{M}_{\zeta}$.
Since $\mathscr{M}_{\zeta}$ is a $\mathbb{Z}[\zeta]$-module it suffices to show that $1/N\in{\mathscr{M}_{\zeta}}$.
In fact, we prove that $\mathbb{Z}[\zeta]/T\subset \mathscr{M}_{\zeta}$ for all $T$ with $\rm{rad}(T)\vert N$.
We compute 
\[
P_{\zeta}(X)=\prod_{1\leq j\leq k}(X-Nx_{j}^{L})=X^{k}-N(a_{k-1}X^{k-1}+a_{k-2}X^{k-2}+\cdots+a_{0}), 
\quad a_{j}\in{\mathbb{Z}}.
\]
By the recursion of Theorem~\ref{elemsatz}, for sufficiently large $n$ we have 
$\scp{\zeta^{n+k}}=-Na_{k-1}\scp{\zeta^{n+k-1}}-Na_{k-2}\scp{\zeta^{n+k-2}}-\cdots-Na_{0}\scp{\zeta^{n}}$
and it follows that the multiplicity $\nu_{p}$ of any prime factor $p\vert N$ satisfies 
\[
\nu_{p}\left(\scp{\zeta^{n+k}}\right)\geq \min\left\{\nu_{p}(\scp{\zeta^{n}}),\nu_{p}(\scp{\zeta^{n+1}}),
\ldots,\nu_{p}(\scp{\zeta^{n+k-1}})\right\}+1.
\]
Hence $\nu_{p}(\scp{\zeta^{n}})$ tends to infinity as $n\to\infty$, verifying that $1/T\in{\mathscr{M}_{\zeta}}$
for any $T$ with $\rm{rad}(T)\vert N$, in particular $1/N\in{\mathscr{M}_{\zeta}}$.

Let $\underline{\chi}=(\chi_{1},\chi_{2},\ldots,\chi_{k})\in{\mathcal{O}_{K}^{k}}$ be 
any integral basis of $\mathcal{O}_{K}$ and 
$B=(b_{i,j})_{1\leq i,j\leq k}\in{\mathbb{Z}^{k\times k}}$ be the matrix representation of 
$(1,\zeta,\zeta^{2},\ldots,\zeta^{k-1})$ in the base
$\underline{\chi}$, i.e., $\zeta^{l}=\sum_{j=1}^{k} b_{l,j}\chi_{j}$ for $0\leq l\leq k-1$.
The fact $\mathbb{Z}[\zeta]/T\subset \mathscr{M}_{\zeta}$ for all $T$ with $\rm{rad}(T)\vert N$
translates in $(B/T)\cdot \underline{\chi}^{t}\in{\mathscr{M}_{\zeta}^{k}}$ for all such $T$.
Since $\mathscr{M}_{\zeta}$ is a $\mathbb{Z}$-module, also 
$\underline{a}\cdot (B/T)\underline{\chi}^{t}\in{\mathscr{M}_{\zeta}}$
for all $\underline{a}=(a_{1},a_{2},\ldots,a_{k})\in{\mathbb{Z}^{k}}$.

By Cramer's rule and since $\zeta$ generates $K$, for any 
canonical basis vector $\underline{e}_{j}$ the equation $\underline{a}B=\det(B)\underline{e}_{j}$ 
has a solution $\underline{a}\in{\mathbb{Z}^{k}}$.
Write $\det(B)=fg$ with $\rm{rad}(f)\vert N$ and $(g,N)=1$. 
Then $gm\equiv 1\bmod N$ has a solution $m\in{\mathbb{Z}}$, write $gm=Nh+1$. It follows
that $\underline{a}mB=m\det(B)\underline{e}_{j}=f(Nh+1)\underline{e}_{j}$ or equivalently 
\[
\underline{a}\cdot \frac{m}{fN}B-h\underline{e}_{j}=\frac{1}{N}\underline{e}_{j}.
\]
Since for $T:=fN$ clearly $\rm{rad}(T)\vert N$ we have 
$\underline{a}\cdot (m/fN)B\cdot \underline{\chi}^{t}\in{\mathscr{M}_{\zeta}}$, but also 
$h\underline{e}_{j}\underline{\chi}^{t}=h\chi_{j}\in{\mathcal{O}_{K}}\subset \mathscr{M}_{\zeta}$, 
so finally $(1/N)\underline{e}_{j}\cdot \underline{\chi}^{t}=(1/N)\chi_{j}\in{\mathscr{M}_{\zeta}}$.
Since this holds for all $1\leq j\leq k$, again using the $\mathbb{Z}$-module property of $\mathscr{M}_{\zeta}$,
we infer $\mathcal{O}_{K}/N\in{\mathscr{M}_{\zeta}}$.
\end{proof}

It remains to investigate if the inclusion 
$\mathcal{O}_{K}\subset \bigcap_{\zeta\in{S_{K}}}\mathscr{M}_{\zeta}$ in~(\ref{eq:unnuetz}) is proper.
We can give an affirmative answer for all quadratic fields.

\begin{theorem} \label{einfall}
Let $K=\mathbb{Q}(\sqrt{d})$ be any real quadratic field. 
Then for any $\zeta\in{S_{K}}=(S\cap K)\setminus{\mathbb{Z}}$, we have $\frac{1}{\sqrt{d_{K}}}\in{\mathscr{M}_{\zeta}}$.
In particular $\mathcal{O}_{K}\subsetneq \bigcap_{\zeta\in{S_{K}}}\mathscr{M}_{\zeta}$.
\end{theorem}

\begin{proof}
For $\zeta\in{S_{K}}$ we have $\Delta(P_{\zeta})=(\zeta-\zeta_{2})^{2}$ for $\zeta_{2}$ the conjugate of $\zeta$.
We compute
\[
\frac{\zeta^{n}}{\sqrt{\Delta(P_{\zeta})}}=\frac{\zeta^{n}}{\zeta-\zeta_{2}}=\frac{\zeta^{n}-\zeta_{2}^{n}}{\zeta-\zeta_{2}}
+\frac{\zeta_{2}^{n}}{\zeta-\zeta_{2}}.
\]
The first expression on the right hand side is a symmetric polynomial in 
$\zeta,\zeta_{2}$ and hence an integer by Theorem~\ref{elemsatz}, 
the second tends to $0$ as $\zeta_{2}\in{(-1,1)}$.
Thus $1/\sqrt{\Delta(P_{\zeta})}\in{\mathscr{M}_{\zeta}}$. However, $\sqrt{\Delta(P_{\zeta})}=N\sqrt{d_{K}}$
by Theorem~\ref{narrer} for an integer $N$ and $d_{K}$ the discriminant of $K$, for which  
$d_{K}\in{\{d,4d\}}$ holds due to Theorem~\ref{kopernik}. 
Thus by the $\mathbb{Z}$-module property of $\mathscr{M}_{\zeta}$ 
also $1/\sqrt{d_{K}}\in{\mathscr{M}_{\zeta}}$ holds, 
whereas $1/\sqrt{d_{K}}\notin{\mathcal{O}_{\mathbb{Q}(\sqrt{d})}}$ because
otherwise $N_{K/\mathbb{Q}(\sqrt{d})}(1/\sqrt{d_{K}})\in{\mathbb{Z}}$ by Theorem~\ref{embed},
however $N_{K/\mathbb{Q}(\sqrt{d})}(1/\sqrt{d_{K}})=-1/d_{K}\notin{\mathbb{Z}}$ 
by Theorem~\ref{narre}.
\end{proof}

Now we want to discuss the inclusion $\mathcal{O}_{K}\subset \bigcap_{\zeta\in{S_{K}}}\mathscr{M}_{\zeta}$
in higher degree $k>2$.
Generalizing the idea of the proof leads to a sufficient criterion for 
$\mathcal{O}_{\mathbb{Q}(\zeta)}\subsetneq \mathscr{M}_{\zeta}$ for any fixed 
$\zeta$ of arbitrary degree. 

\begin{proposition} \label{beispuel}
Let $\zeta\in{S}$ of dergee $k$ with conjugates $\zeta_{1}=\zeta,\zeta_{2},\zeta_{3},\ldots,\zeta_{k}$.
Then
\[
\Psi(\zeta):=\sum_{j=2}^{k}\frac{1}{\zeta-\zeta_{j}}\in{\mathscr{M}_{\zeta}}
\]
holds. In particular, $\Psi(\zeta)\notin{\mathcal{O}_{\mathbb{Q}(\zeta)}}$ is sufficient for 
$\mathcal{O}_{\mathbb{Q}(\zeta)}\subsetneq \mathscr{M}_{\zeta}$.
\end{proposition}

\begin{proof}
By Theorem~\ref{elemsatz}
\[
\sum_{1\leq i<j\leq k} \frac{\zeta_{i}^{n}-\zeta_{j}^{n}}{\zeta_{i}-\zeta_{j}}=
\sum_{2 \leq j\leq k} \frac{\zeta^{n}}{\zeta-\zeta_{j}}-
\sum_{2 \leq j\leq k} \frac{\zeta_{j}^{n}}{\zeta-\zeta_{j}}+
\sum_{2\leq i<j\leq k} \frac{\zeta_{i}^{n}-\zeta_{j}^{n}}{\zeta_{i}-\zeta_{j}}
\]
is an integer. The second and third sum on the right hand side tend to $0$ for $n\to\infty$ 
because $\vert \zeta_{j}\vert<1$ for $2\leq j\leq k$, which readily implies the assertion.
\end{proof}

\begin{remark}
Proposition~\ref{beispuel} and~(\ref{eq:ungarturm}) imply $\Psi(\zeta)\in{\mathbb{Q}(\zeta)}$ for $\zeta\in{S}$.
This is easily verified for any $\zeta$ algebraic of degree at most $3$, but it seems not trivial 
for $k\geq 4$.
\end{remark}

\begin{remark}
Even if for a fixed field $K$ and all $\zeta\in{S_{K}}$ we have 
$\mathcal{O}_{\mathbb{Q}(\zeta)}\subsetneq \mathscr{M}_{\zeta}$, which is guaranteed if
$\Psi(\zeta)\notin{\mathcal{O}_{\mathbb{Q}(\zeta)}}$, 
there might not be a common (uniform) element for all $\zeta\in{S_{K}}$ as in Theorem~\ref{einfall}.  
\end{remark}

\subsection{Equality vs proper inclusion in~(\ref{eq:ungarturm}),(\ref{eq:ungerturm}). 
When is $\mathscr{M}_{\zeta}$ a ring}  \label{letztmals}

\begin{proposition} \label{selchfleisch}
Let $\theta\in{\mathcal{O}}$. We have $\mathbb{Z}[\theta]= \mathbb{Z}[\theta,\theta^{-1}]$
if and only if $\theta\in{\mathcal{O}^{\ast}}$, as well as 
$\mathcal{O}_{\mathbb{Q}(\theta)}=\mathcal{O}_{\mathbb{Q}(\theta)}[\theta^{-1}]$
if and only if $\theta\in{\mathcal{O}^{\ast}}$.
\end{proposition}

\begin{proof}
If $X^{k}+a_{k-1}X^{k-1}+\cdots+a_{1}X+a_{0}$ is the minimal
polynomial of $\theta$, then 
\begin{equation} \label{eq:sotschi}
\theta^{-1}=-\frac{\theta^{k-1}+a_{k-1}\theta^{k-2}+\cdots+a_{1}}{a_{0}}.
\end{equation}
The representation of $\theta^{-1}$ in the form~(\ref{eq:sotschi}) is unique because
$\{1,\theta,\theta^{2},\ldots,\theta^{k-1}\}$ is a base of $\mathbb{Q}(\theta)$ as a $\mathbb{Q}$-vector space.
Thus $\vert a_{0}\vert=1$ is necessary and sufficient for $\theta^{-1}\in{\mathbb{Z}[\theta]}$,
or equivalently $\mathbb{Z}[\theta]= \mathbb{Z}[\theta,\theta^{-1}]$. 
Similarly, $\mathcal{O}_{\mathbb{Q}(\theta)}=\mathcal{O}_{\mathbb{Q}(\zeta)}[\theta^{-1}]$ holds
if and only if $\theta^{-1}\in{\mathcal{O}_{\mathbb{Q}(\theta)}}$, or equivalently
$\theta\in{\mathcal{O}^{\ast}}$.
\end{proof}

We deduce two Corollaries.

\begin{corollary}  \label{nichtgleich}
If $\zeta\in{S^{\ast}}$ is quadratic, we have 
$\mathcal{O}_{\mathbb{Q}(\zeta)}[\zeta^{-1}]\subsetneq \mathscr{M}_{\zeta}$.
\end{corollary}

\begin{proof}
Ont he one hand $\mathcal{O}_{\mathbb{Q}(\zeta)}=\mathcal{O}_{\mathbb{Q}(\zeta)}[\zeta^{-1}]$ holds
by Proposition~\ref{selchfleisch}, on the other hand $\mathcal{O}_{\mathbb{Q}(\zeta)}\subsetneq \mathscr{M}_{\zeta}$
holds by Theorem~\ref{einfall}.
\end{proof}

\begin{corollary} \label{zamgfasst}
For $\theta\in{\mathcal{O}^{\ast}}$ either {\upshape(A)} or {\upshape(B)} applies:
\begin{eqnarray*}
&(A)&\qquad \mathbb{Z}[\theta]=\mathbb{Z}[\theta,\theta^{-1}]\subsetneq \mathcal{O}_{\mathbb{Q}(\theta)}
=\mathcal{O}_{\mathbb{Q}(\theta)}[\theta^{-1}]   \\
&(B)& \qquad \mathbb{Z}[\theta]=\mathcal{O}_{\mathbb{Q}(\theta)}= \mathbb{Z}[\theta,\theta^{-1}]
=\mathcal{O}_{\mathbb{Q}(\theta)}[\theta^{-1}].
\end{eqnarray*}
\end{corollary}

\begin{proposition}  \label{pakistan}
Let $\theta\in{\mathcal{O}}$ with minimal polynomial $P_{\theta}$.
Any element $a$ in $\mathbb{Z}[\theta,\theta^{-1}]$ can be uniquely written as
$a=(r_{k-1}\zeta^{k-1}+r_{k-2}\zeta^{k-2}+\cdots+r_{0})/N$
with $\rm{rad}(N)\vert \rm{rad}(P_{\theta}(0))$.
\end{proposition}   

\begin{proof}
Every element $a\in{\mathbb{Z}[\theta,\theta^{-1}]}$ can obviously be written as $R(\theta)/\theta^{n}$
for $R\in{\mathbb{Z}[X]}, n\geq 0$ an integer. By~(\ref{eq:sotschi}) we have 
$\theta^{-n}=(1/P_{\theta}(0)^{l})S(\theta)$ for $S\in{\mathbb{Z}[X]}, l\geq 0$ an integer. 
Thus $a=(1/P_{\theta}(0)^{l})\cdot R(\theta)S(\theta)$.
Reducing $R(\theta)S(\theta)$ mod $P_{\zeta}$ we obtain $a=(1/P_{\theta}(0)^{l})T(\theta)$ for a polynomial 
$T\in{\mathbb{Z}[X]}$ of degree at most $k-1$. Thus reducing common factors of the gcd of the coefficients of $T$
and $P_{\theta}(0)^{l}$, we obtain the unique given form of $a$, and clearly 
the denominator has the prescribed property. The uniqueness follows by a vector space argument.
\end{proof}

We turn towards the inclusions 
$\mathbb{Z}[\zeta]\subset \mathcal{O}_{\mathbb{Q}(\zeta)}$,
$\mathbb{Z}[\zeta,\zeta^{-1}]\subset \mathcal{O}_{\mathbb{Q}(\zeta)}[\zeta^{-1}]$ 
in~(\ref{eq:ungarturm}).

\begin{corollary} \label{liniz}
Let $\zeta\in{S}$. A sufficient condition for 
$\mathbb{Z}[\zeta,\zeta^{-1}]\subsetneq \mathscr{M}_{\zeta}$
is $\rm{rad}(\Delta(P_{\zeta}))\nmid \rm{rad}(P_{\zeta}(0))$, i.e.,
there exists a prime $q$ with $q\vert\Delta(P_{\zeta})$
but $q\nmid P_{\zeta}(0)$. In particular, for any $\zeta\in{S^{\ast}}$ we have
\[
\mathbb{Z}[\zeta]=\mathbb{Z}[\zeta,\zeta^{-1}]\subsetneq \mathscr{M}_{\zeta}.
\]
\end{corollary}

\begin{proof}
Assume $q$ has the properties $q\vert\Delta(P_{\zeta}), q\nmid P_{\zeta}(0)$.
Recalling Proposition~\ref{pakistan}, all elements in reduced form in $\mathbb{Z}[\zeta,\zeta^{-1}]$ have denominators 
divisible only by primes dividing $P_{\zeta}(0)$, so certainly not by $q$. 
However, Theorem~\ref{kamou} asserts the existence of a reduced element in $\mathscr{M}_{\zeta}$  
with denominator divisible by $q$, so the inclusion is strict.

The specialization follows because Pisot units have $\vert P_{\zeta}(0)\vert=1$ but
$\vert\Delta(P_{\zeta})\vert>1$ by Theorem~\ref{narrer}, 
so we may take $q$ any prime divisor of $\Delta(P_{\zeta})$. The equality is due to Proposition~\ref{selchfleisch}. 
\end{proof}

Though we have $\mathbb{Z}[\zeta]\subset\mathscr{M}_{\zeta}$,
Corollary~\ref{liniz} implies $\mathscr{M}_{\zeta}\neq\mathbb{Z}[\zeta]$ for any $\zeta\in{S}$.

\begin{corollary}  \label{mrmr}
For any $\zeta\in{S}$, at least one of the inclusions 
$\mathcal{O}_{\mathbb{Q}(\zeta)}\subset \mathscr{M}_{\zeta}$,
$\mathbb{Z}[\zeta,\zeta^{-1}]\subset \mathscr{M}_{\zeta}$
is proper. In particular $\mathbb{Z}[\zeta]\subsetneq \mathscr{M}_{\zeta}$.
\end{corollary}

\begin{proof}
If $\zeta$ is a unit it follows from Corollary~\ref{liniz}, else from Proposition~\ref{selchfleisch}.
The specification follows from the first inclusion in~(\ref{eq:ungarturm}),(\ref{eq:ungerturm}).
\end{proof}

\begin{corollary} \label{wannring}
For $\zeta\in{S^{\ast}}$, the set $\mathscr{M}_{\zeta}$ is a ring if and only if
\begin{equation} \label{eq:allesda}
\mathbb{Z}[\zeta]=\mathbb{Z}[\zeta,\zeta^{-1}]\subsetneq \mathcal{O}_{\mathbb{Q}(\zeta)}
=\mathcal{O}_{\mathbb{Q}(\zeta)}[\zeta^{-1}]=\mathscr{M}_{\zeta},
\end{equation}
in particular $\mathcal{O}_{\mathbb{Q}(\zeta)}= \mathscr{M}_{\zeta}$ and 
$\rm{rad}(\Delta(P_{\zeta}))<\vert\Delta(P_{\zeta})\vert$ is necessary.
\end{corollary}

\begin{proof}
Obviously, if $\mathcal{O}_{\mathbb{Q}(\zeta)}= \mathscr{M}_{\zeta}$ then $\mathscr{M}_{\zeta}$ is a ring,
and in this case Corollary~\ref{mrmr} implies (A) in Corollary~\ref{zamgfasst}, proving~(\ref{eq:allesda}).
According to Theorem~\ref{maximalordnung} the ring of integers $\mathcal{O}_{\mathbb{Q}(\zeta)}$ is
the maximal order of $\mathbb{Q}(\zeta)$, i.e., there is no strictly larger finitely generated $\mathbb{Z}$-submodule 
of $\mathbb{Q}(\zeta)$ which is a ring. Since $\mathscr{M}_{\zeta}$ is a finitely generated
$\mathbb{Z}$-module by Corollary~\ref{ontothetop} which contains $\mathcal{O}_{\mathbb{Q}(\zeta)}$
by~(\ref{eq:ungarturm}), assuming strict inclusion $\mathcal{O}_{\mathbb{Q}(\zeta)}\subsetneq \mathscr{M}_{\zeta}$
it cannot be a ring, again leading to $\mathcal{O}_{\mathbb{Q}(\zeta)}=\mathscr{M}_{\zeta}$ and thus~(\ref{eq:allesda}).
Finally, if $\Delta(P_{\zeta})$ was square-free we could not have (A) due to Theorem~\ref{narrer}.
\end{proof}

By Theorem~\ref{einfall} we cannot have~(\ref{eq:allesda}) for quadratic $\zeta\in{S^{\ast}}$.
It is reasonable to believe there is actually no $\zeta\in{S^{\ast}}$ where~(\ref{eq:allesda}) holds,
such that actually $\mathcal{O}_{\mathbb{Q}(\zeta)}\subsetneq \mathscr{M}_{\zeta}$ for all $\zeta\in{S}$,
which would improve Corollary~\ref{mrmr}.
See also Proposition~\ref{beispuel}.

We haven't presented an example yet of a $\zeta\in{S\setminus{S^{\ast}}}$ where $\mathscr{M}_{\zeta}$ is no ring.
This is usually the case, we just give one example.

\begin{proposition}
For $\zeta=2+\sqrt{6}$ the Pisot root of $X^{2}-4X-2$ the set $\mathscr{M}_{\zeta}$ is no ring.
\end{proposition}

\begin{proof}
We compute $P_{\zeta,3}=(X+1)^{2}$.
By Lemma~\ref{daslemma} we have $(\zeta+1)/3=(3+\sqrt{6})/3\in{\mathscr{M}_{\zeta}}$.
From $\mathscr{M}_{\zeta}$ being a $\mathbb{Z}$-module it follows 
$(3+\sqrt{6})/3-1=\sqrt{2/3}\in{\mathscr{M}_{\zeta}}$ as well. 
However, $\sqrt{2/3}\cdot \sqrt{2/3}=2/3\notin{\mathscr{M}_{\zeta}}$ because
it is easy to see the period of $\scp{\zeta^{n}}\bmod 3$ is $\overline{1,2}$, so no number
of the form $(2/3)\cdot\scp{\zeta^{n}}$ is an integer for sufficiently large $n$.
\end{proof}

\section{Applications to integral bases}

We want to point out that by $\mathcal{O}_{\mathbb{Q}(\zeta)}\subset \mathscr{M}_{\zeta}$,
combination of Lemma~\ref{daslemma}, Theorem~\ref{reellfeld} and Theorem~\ref{okkident}
gives some information on the form of the ring of integers of real
number fields. Corollary~\ref{gring} can be helpful to find an integral basis of a real field $K$,
provided that one can easily determine a Pisot number in $K$. To obtain
the best possible result we will use the following Lemma~\ref{christ}. To simplify its proof we
prepend a proposition.

\begin{proposition} \label{algeig}
Let $P\in{\mathbb{Z}[X]}$ of degree $k$ be monic and $Q\in{\mathbb{Z}[X]}$ of 
degree $r<k$ with relatively prime coefficients.
Then there exists $R\in{\mathbb{Z}[X]}$ such that $Q(X)\cdot R(X) \bmod P(X)$, written as
$b_{k-1}X^{k-1}+b_{k-2}X^{k-2}+\cdots+b_{0}$, has $b_{k-1}=1$. 
\end{proposition}

\begin{proof}
Writing $P(X)=X^{k}+a_{k-1}X^{k-1}+\cdots+a_{0}$ and 
$Q(X)=b_{r}X^{r}+b_{r-1}X^{r-1}+\cdots+b_{0}$, we look
at coefficient of $X^{k-1}$ of $X^{t}Q(X)\bmod P(X)$ for $t\geq k-1-r\geq 0$, so say
$X^{t}Q(X)\bmod P(X)=c_{t}X^{k-1}+H(X)$ with $H$ of degree at most $k-2$. 
If the coefficients $\{c_{t}: t\geq k-r-1\}$ are relatively prime, then clearly we can find a 
linear combination $R(X)=\sum_{j=0}^{m} d_{j}X^{j}$ such that $Q(X)R(X)\bmod P(X)$ starts with $X^{k-1}$
and we are done. So suppose there is a common prime divisor $p$ of $\{c_{t}: t\geq k-r-1\}$. 
Taking $t=k-1-r$ gives $p\vert b_{r}$. Taking $t=k-r$ we get $p\vert c_{k-r}=b_{r-1}-a_{k-1}b_{r}$, so in combination
with $p\vert b_{r}$ we obtain $p\vert b_{r-1}$ as well. It is easy to see that in general we can write
$c_{k-r-1+u}=b_{r-u}-\sum_{j=0}^{u-1} e_{j}b_{r-j}$ for all $u\geq 0$
with integers $e_{j}$ that are linear combinations of 
products of some $a_{.}, b_{.}$. Thus $p\vert b_{r}, p\vert b_{r-1},\ldots,p\vert b_{r-u+1}$ implies
$p\vert b_{r-u}$ for any $u\geq 0$, so eventually $p$ divides all coefficients of $Q$, contradicting the assumption.
\end{proof}

\begin{lemma} \label{christ}
Let $\theta$ be an algebraic integer of degree $k$ with minimal polynomial $P_{\theta}$, 
whose discriminant is $\Delta(P_{\theta})$.
Further let $Q\in{\mathbb{Z}[X]}$ of degree $r<k$
arbitrary with relatively prime coefficients and $N\in{\mathbb{Z}\setminus{\{0\}}}$.
Then $Q(\theta)/N\in{\mathcal{O}_{\mathbb{Q}(\theta)}}$ implies $N\vert \Delta(P_{\theta})$.
\end{lemma}

\begin{proof}
Say $\theta=:\theta_{1}$ and write
$P_{\theta}(X)=(X-\theta_{1})(X-\theta_{2})\cdots (X-\theta_{k})$.
First assume $Q$ is monic, say $Q(X)=(X-\alpha_{1})(X-\alpha_{2})\cdots(X-\alpha_{r})$.
Then 
\[
a:=\frac{Q(\theta)}{N}=\frac{(\theta-\alpha_{1})(\theta-\alpha_{2})\cdots (\theta-\alpha_{r})}{N}
\in{\mathcal{O}_{\mathbb{Q}(\theta)}}
\]
implies
\begin{equation} \label{eq:lampard}
\frac{(\theta_{i}-\alpha_{1})(\theta_{i}-\alpha_{2})\cdots (\theta_{i}-\alpha_{r})}{N}
\in{\mathcal{O}_{\mathbb{Q}(\theta_{i})}}, \qquad 1\leq i\leq k.
\end{equation}
Indeed, $b:=Na=(\theta-\alpha_{1})(\theta-\alpha_{2})\cdots (\theta-\alpha_{r})$ implies
$\sigma_{i}(b)=\sigma_{i}(Na)=N\sigma_{i}(a)$, with $\sigma_{i}:\mathbb{Q}(\theta)\mapsto \mathbb{Q}(\theta_{i})$ 
the morphism mapping $\theta\to \theta_{i}$, and 
$\sigma_{i}(b)=\sigma_{i}(Q(\theta))=Q(\sigma_{i}(\theta))=Q(\theta_{i})$ obviously 
equals the nominator in~(\ref{eq:lampard}) and $\sigma_{i}(a)\in{\mathcal{O}_{\mathbb{Q}(\theta_{i})}}$.
Let $L:=\mathbb{Q}(\theta_{1},\ldots,\theta_{k},\alpha_{1},\ldots,\alpha_{r})$ 
be the splitting field of $P_{\theta}(X)Q(X)$. Say $(N)=\mathscr{P}_{1}^{n_{1}}\cdots \mathscr{P}_{t}^{n_{t}}$
is the prime ideal decomposition of $(N)=\mathcal{O}_{L}N$ over $\mathcal{O}_{L}$. 
Consider $\mathscr{P}_{j}$ for an arbitrary but fixed index $j$.
In view of~(\ref{eq:lampard}), for any $1\leq i\leq k$ the prime ideal 
$\mathscr{P}_{j}$ occurs at least $n_{j}$ times in the product 
$\prod_{1\leq m\leq r}\mathcal{O}_{L}(\theta_{i}-\alpha_{m})$ of principal ideals. 
Due to $r<k$, it follows from repeated use of pigeon hole principle that there have to be $n_{j}$ triples
$(i_{0},i_{1},r_{0})$, $1\leq i_{0}< i_{1}\leq k$, counted with multiplicities, with $\mathscr{P}_{j}$ dividing both 
$\mathcal{O}_{L}(\theta_{i_{0}}-\alpha_{r_{0}}), \mathcal{O}_{L}(\theta_{i_{1}}-\alpha_{r_{0}})$. 
For any such triple it follows that both $\theta_{i_{0}}-\alpha_{r_{0}},\theta_{i_{1}}-\alpha_{r_{0}}$ are in
$\mathscr{P}_{j}$, hence $\theta_{i_{0}}-\theta_{i_{1}}\in{\mathscr{P}_{j}}$ too, or equivalently
$\mathscr{P}_{j}\vert \mathcal{O}_{L}(\theta_{i_{0}}-\theta_{i_{1}})$. We conclude
$\mathscr{P}_{j}^{n_{j}}\vert (\Delta(P_{\theta}))$, and 
since $j$ was arbitrary, $(N)\vert (\Delta(P_{\theta}))$. 
However, as $N,\Delta(P_{\theta})$ are rational integers, this easily implies $N\vert \Delta(P_{\theta})$
over $\mathbb{Z}$, see~(iii) of Theorem~\ref{dastheorem}.

Finally, we have to show the assumption on $Q$ to be monic is no restriction. For the given $Q(X)$ 
by Proposition~\ref{algeig} with $P:=P_{\theta}$
we can find $R(X)\in{\mathbb{Z}[X]}$ such that $S(X):=Q(X)R(X) \bmod P_{\theta}(X)$ is monic. 
However, from $\theta\in{\mathcal{O}_{\mathbb{Q}(\theta)}}$ it follows 
that $R(\theta)\in{\mathcal{O}_{\mathbb{Q}(\theta)}}$, 
so we see if $Q(\theta)/N\in{\mathcal{O}_{\mathbb{Q}(\theta)}}$ also 
$S(\theta)/N\in{\mathcal{O}_{\mathbb{Q}(\theta)}}$, and the results from the monic case apply.
\end{proof}

\begin{remark}
In general, it is not possible to restrict to primes dividing the discriminant $d_{K}$
of the field instead of the polynomial discriminant. Table~1 on page~8 in~\cite{spearman}
provides counterexamples. For instance $P(X)=X^{3}+30X+90$, where an integral basis of $K=\mathbb{Q}(\theta)$
for $\theta$ a root of $P$ is given by $\{1,\theta,(20-4\theta+\theta^{2})/13\}$ 
but $13\nmid -3^{3}5^{2}=d_{K}$, whereas indeed $\Delta(P)=13^{2}d_{K}$.
\end{remark}

\begin{corollary} \label{gring}
Let $K$ be a real number field of degree $[K:\mathbb{Q}]=k\geq 2$ and 
$\zeta\in{S_{K}}$ arbitrary with Pisot polynomial $P_{\zeta}$. 
Any integral basis of $K$ can be written as 
\begin{equation} \label{eq:totale}
\frac{1}{N}\{ T_{1}(\zeta), T_{2}(\zeta), \ldots, T_{k}(\zeta)\} 
\end{equation}
for $T_{j}\in{\mathbb{Z}[X]}$ of degree at most $(k-1)$ and
$N$ a divisor of the discriminant $\Delta(P_{\zeta})$. 
Moreover, writing $T_{j}(\zeta)/N=U_{j}(\zeta)/N_{j}$ with coefficients of $U_{j}$
and $N_{j}$ relatively prime {\upshape(}division by gcd of the coefficients of $T_{j}$ 
and $N${\upshape)}, for all primes $q$ dividing $N_{j}$ we have $R_{\zeta,q}\vert U_{j,q}$,
where $R_{\zeta,q}$ is defined in Lemma~\ref{daslemma}
and $U_{j,q}=\pi_{q}(U_{j})$ is the reduction of $U_{j}$ over $\mathbb{Z}_{q}[X]$.
In particular $R_{\zeta,q}\vert T_{j,q}$.
\end{corollary}

\begin{proof}
By Theorem~\ref{reellfeld} the real number field $K$ contains $\zeta\in{S_{K}}$. 
Since $\{1,\zeta,\ldots,\zeta^{k-1}\}$ 
spans $\mathbb{Q}(\zeta)$, every element $a\in{\mathcal{O}_{\mathbb{Q}(\zeta)}}$
has the form $a=Q(\zeta)/d$ with $Q\in{\mathbb{Z}[X]}$ of degree $r<k$ and $d\in{\mathbb{Z}}$, and 
$d\vert \Delta(P_{\zeta})$ by Lemma~\ref{christ}. 
Hence taking common denominators of any integral basis with elements in reduced form, 
it can be written as in~(\ref{eq:totale}) with $N\vert \Delta(P_{\zeta})$. 
The assertion on the shape of the $U_{j}$ follows from Lemma~\ref{daslemma}
due to the facts that $\mathcal{O}_{K}=\mathcal{O}_{\mathbb{Q}(\zeta)}\subset \mathscr{M}_{\zeta}$ and
if $U_{j}(\zeta)/N_{j}\in{\mathscr{M}_{\zeta}}$
then also $(N_{j}/q)\cdot (U_{j}(\zeta)/N_{j})=U_{j}(\zeta)/q\in{\mathscr{M}_{\zeta}}$ 
by the $\mathbb{Z}$-module property. The last point follows from 
$T_{j}(X)=b_{j}U_{j}(X)$ for some $b_{j}\in{\mathbb{Z}}$.
\end{proof}

\begin{remark}
The importance of Corollary~\ref{gring} is mostly reversing the statement, i.e., having found
an arbitrary Pisot number $\zeta\in{K}$ one has stringent restrictions on the shape of an integral
basis of $\mathcal{O}_{K}$.
Concerning how to determine $\zeta\in{S_{K}}$, we should mention 
that the proof of Theorem~\ref{reellfeld} in~\cite{27} uses 
Minkowski's Theorem and thus is not constructive, though.
\end{remark}

\begin{remark}
The restriction of $\zeta$ being a Pisot unit/number seems unnecessary,
the properties of symmetric polynomials should 
imply the analogue result for any algebraic integer which is a primitive element of the field $K$.
\end{remark}

Prior work on algorithms to determine integral bases of a number field can be
found for example in~\cite{hoeij}.

\end{document}